\documentclass[a4paper]{amsart}

\usepackage[OT2,T1]{fontenc}
\DeclareSymbolFont{cyrletters}{OT2}{wncyr}{m}{n}
\DeclareMathSymbol{\Sha}{\mathalpha}{cyrletters}{"58}

\usepackage{amssymb,amscd}
\usepackage{mathrsfs}

\usepackage[
backref,
pdfauthor={Han Wu},  
pdftitle={Ch\^atelet surfaces and non-invariance of the Brauer-Manin obstruction for $3$-folds},
]{hyperref}
\usepackage[alphabetic,backrefs,lite]{amsrefs}  

\theoremstyle{plain}
\theoremstyle{definition}

\newtheorem{theorem}{Theorem}[subsection]
\newtheorem{thm}{Theorem}[subsubsection]
\newtheorem{lemma}[theorem]{Lemma}
\newtheorem{corollary}[theorem]{Corollary}

\newtheorem{proposition}[theorem]{Proposition}
\newtheorem{prop}[theorem]{Proposition}

\theoremstyle{definition}
\newtheorem{definition}[theorem]{Definition}

\newtheorem{qu}[theorem]{Question}

\newtheorem{eg}[theorem]{Example}

\newtheorem{conjecture}[theorem]{Conjecture}

\newtheorem{exceptions}[thm]{Exceptions}

\theoremstyle{remark}
\newtheorem{remark}[theorem]{Remark}

\newtheorem{recall}[theorem]{Recall}

\renewcommand{\AA}{\mathbb{A}}

\newcommand{\QQ}{\mathbb{Q}}
\newcommand{\RR}{\mathbb{R}}
\newcommand{\ZZ}{\mathbb{Z}}

\newcommand{\FF}{\mathbb{F}}
\newcommand{\GG}{\mathbb{G}}
\newcommand{\PP}{\mathbb{P}}

\newcommand{\Ecal}{{\mathcal E}}
\newcommand{\Lcal}{{\mathcal L}}
\newcommand{\Ocal}{{\mathcal O}}

\newcommand{\Lscr}{{\mathscr L}}

\newcommand{\half}{1/2}

\newcommand{\inuap}{\inv_{v'}(A(P_{v'}))}
\newcommand{\invap}{\inv_v(A(P_v))}
\newcommand{\invaq}{\inv_v(A(Q_v))}
\newcommand{\invaqo}{\inv_v(A(Q_0))}

\DeclareMathOperator{\Gal}{Gal}

\DeclareMathOperator{\inv}{inv}
\DeclareMathOperator{\Proj}{Proj}
\DeclareMathOperator{\Sym}{Sym}
\DeclareMathOperator{\Spec}{Spec}

\newcommand{\defi}[1]{\textsf{#1}} 

\newcommand{\Br}{\textup{Br}}
\newcommand{\et}{\textup{\'et}}

\makeatletter
\g@addto@macro\bfseries{\boldmath}  
\makeatother

\begin{document}
	
	\begin{title}
		{Ch\^atelet surfaces and non-invariance of the Brauer-Manin obstruction for $3$-folds}  
	\end{title}
	\author{Han Wu}
	\address{University of Science and Technology of China,
		School of Mathematical Sciences,
		No.96, JinZhai Road, Baohe District, Hefei,
		Anhui, 230026. P.R.China.}
	\email{wuhan90@mail.ustc.edu.cn}
	\date{}




	\begin{abstract} 
		In this paper, we construct three kinds of Ch\^atelet surfaces, which have some given arithmetic properties with respect to field extensions of number fields. We then use these constructions to study the properties of weak approximation with Brauer-Manin obstruction and the Hasse principle with Brauer-Manin obstruction for $3$-folds, which are pencils of Ch\^atelet surfaces parameterized by a curve, with respect to field extensions of number fields. We give general constructions (conditional on a conjecture of M. Stoll) to negatively answer some questions, and illustrate these constructions and some exceptions with some explicit unconditional examples.
	\end{abstract} 
	
	\maketitle

	\section{Introduction}
	
	\subsection{Background}
	Let $X$ be a proper algebraic variety defined over a number field $K.$ Let $\Omega_K$ be the set of all nontrivial places of $K.$ Let $\infty_K\subset \Omega_K$ be the subset of all archimedean places, and let $S\subset \Omega_K$ be a finite subset. Let $K_v$ be the completion of $K$ at $v\in \Omega_K.$ Let $\AA_K$ be the ring of ad\`eles of $K.$ If the set of $K$-rational points $X(K)\neq\emptyset,$ then the set of adelic points $X(\AA_K)\neq\emptyset.$  The converse is known as the Hasse principle. We say that $X$ is a \defi{counterexample to the Hasse principle} if $X(\AA_K)\neq\emptyset$ whereas $X(K)=\emptyset.$ A well know counterexample to the Hasse principle is Selmer's cubic curve defined over $\QQ$ by
	$3w_0^3+4w_1^3+5w_2^3=0$ with homogeneous coordinates $(w_0:w_1:w_2)\in \PP^2.$ Let $pr^S\colon\AA_K\to \AA_K^S$ be the natural projection of the ring of ad\`eles and ad\`eles without $S$ components, which induces a natural projection $pr^S\colon X(\AA_K)\to X(\AA_K^S)$ if $X(\AA_K)\neq \emptyset.$ For $X$ is proper, the set of adelic points $X(\AA_K^S)$ is equal to $\prod_{v\in \Omega_K\backslash S}X(K_v),$ and the adelic topology of $X(\AA_K^S)$  is indeed the product topology of $v$-adic topologies. 
	Viewing $X(K)$ as a subset of $X(\AA_K)$ (respectively of $X(\AA_K^S)$) by the diagonal embedding, we say that $X$
	satisfies \defi{weak approximation} (respectively \defi{weak approximation off $S$}) if $X(K)$ is dense in $X(\AA_K)$ (respectively in $X(\AA_K^S)$), cf. \cite[Chapter 5.1]{Sk01}. Cohomological obstructions have been used to explain failures of the Hasse principle and nondensity of $X(K)$ in $X(\AA_K^S).$  Let $\Br(X)=H^2_{\et}(X,\GG_m)$ be the Brauer group of $X.$ The Brauer-Manin pairing
	$$X(\AA_K)\times\Br(X)\to \QQ/\ZZ,$$ 
	suggested by Brauer-Manin \cite{Ma71}, between $X(\AA_K)$ and $\Br(X),$ is provided by local class field theory. The left kernel of this pairing is denoted by $X(\AA_K)^{\Br},$ which is a closed subset of $X(\AA_K).$ By the global reciprocity in class field theory, there is an exact sequence: 
	$$0\to \Br(K)\to \bigoplus\limits_{v\in \Omega_K} \Br(K_v)\to \QQ/\ZZ\to 0.$$
	It induces an inclusion: $X(K)\subset pr^S(X(\AA_K)^\Br).$ We say that the failure of the Hasse principle of $X$ is \defi{explained by the Brauer-Manin obstruction} if  $X(\AA_K)\neq\emptyset$ and $X(\AA_K)^{\Br}=\emptyset.$ For the failure of the Hasse principle of a smooth, projective, and geometrically connected curve over $K,$ if the Tate-Shafarevich group and the rational points set of its Jacobian are finite, then this failure can be explained by the Brauer-Manin obstruction, cf. \cite{Sc99}. A counterexample to the Hasse principle such that its failure of the Hasse principle is not explained by the Brauer-Manin obstruction, was constructed firstly by Skorobogatov \cite{Sk99}, subsequently by Poonen \cite{Po10}, Harpaz and Skorobogatov \cite{HS14}, Colliot-Th\'el\`ene, P\'al and Skorobogatov \cite{CTPS16} and so on.
	We say that $X$ satisfies \defi{weak approximation with Brauer-Manin obstruction} (respectively \defi{with Brauer-Manin obstruction off $S$}) if $X(K)$ is dense in $X(\AA_K)^\Br$ (respectively in $pr^S(X(\AA_K)^\Br)$). For an elliptic curve defined over $\QQ,$ if its analytic rank equals one, then this elliptic curve satisfies weak approximation with Brauer-Manin obstruction, cf. \cite{Wa96}. For an abelian variety defined over $K,$ if its Tate-Shafarevich group is finite, then this abelian variety satisfies weak approximation with Brauer-Manin obstruction off $\infty_K,$ cf. \cite[Proposition 6.2.4]{Sk01}. 
	For any smooth, proper and rationally connected variety $X$ defined over an number field, it is conjectured by J.-L. Colliot-Th\'el\`ene that $X$ satisfies weak approximation with Brauer-Manin obstruction. The Colliot-Th\'el\`ene's conjecture holds for Ch\^atelet surfaces, cf. \cite{CTSSD87a,CTSSD87b}.
	For any smooth, projective, and geometrically connected curve $C$ defined over an number field $K,$ it is conjectured by Stoll \cite{St07} that $C$ satisfies weak approximation with Brauer-Manin obstruction off $\infty_K\colon$ see Conjecture \ref{conjecture Stoll} for more details.
	Before the Stoll's conjecture, if this curve $C$ is a counterexample to the Hasse principle,
	Skorobogatov \cite[Chapter 6.3]{Sk01} asked a weaker open question: is the failure of the Hasse principle of $C$ explained by the Brauer-Manin obstruction?
	
	\subsection{Questions}\label{Questions}
	Given a nontrivial extension of number fields $L/K,$ and a finite subset $S\subset\Omega_K,$ let $S_L\subset \Omega_L$ be the subset of all places above $S.$ Let $X$ be a smooth, projective, and geometrically connected variety defined over $K.$ Let $X_L=X\times_{\Spec K} {\Spec L}$ be the base change of $X$ by $L.$ In this paper, we consider the following questions.
	\begin{qu}\label{question on WA1}
		If the variety $X$ has a $K$-rational point, and satisfies weak approximation with Brauer-Manin obstruction off $S,$ must $X_L$ also satisfy weak approximation with Brauer-Manin obstruction off $S_L?$	
	\end{qu}
	\begin{qu}\label{question on HP}		
		Assume that the varieties $X$ and $X_L$ are counterexamples to the Hasse principle. If the failure of the Hasse principle of $X$ is explained by the Brauer-Manin obstruction, must the failure of the Hasse principle of $X_L$ also be explained by the Brauer-Manin obstruction?		
	\end{qu}
	
	\subsection{Main results for Ch\^atelet surfaces}\label{introduction Main results for Ch\^atelet surfaces}
	A Ch\^atelet surface defined over $\QQ,$ which is a counterexample to the Hasse principle, was constructed by Iskovskikh \cite{Is71}. Poonen \cite{Po09} generalized it to any given number field. For any number field $K,$ Liang \cite{Li18} constructed a Ch\^atelet surface defined over $K,$ which has a $K$-rational point and does not satisfy weak approximation off $\infty_K.$ By using weak approximation and strong approximation off all $2$-adic places for $\AA^1$ (cf. Lemma \ref{lemma strong approximation for A^1}) to choose elements in $K,$ and using \v{C}ebotarev's density theorem (cf. Theorem \ref{theorem Chebotarev density theorem}) to add some splitting conditions, we construct three kinds of Ch\^atelet surfaces to generalize their arguments.
	
	\begin{proposition}	
		For any extension of number fields $L/K,$ and any finite subset $S \subset \Omega_K\backslash \{$all complex and $2$-adic places$\}$ splitting completely in $L,$ 
		there exist Ch\^atelet surfaces $V_1, V_2, V_3$ defined over $K,$ which have the following properties.
		\begin{itemize}
			\item  The subset $V_1(\AA_K^S)\subset V_1(\AA_L^{S_L})$ is nonempty, but $V_1(K_v)=V_1(L_{v'})=\emptyset$  for all $v\in S$ and all $v'\in S_L,$ cf. Proposition \ref{proposition no local point for S}.
			\item The Brauer group $\Br(V_2)/\Br(K)\cong\Br(V_{2L})/\Br(L)\cong \ZZ/2\ZZ,$ is generated by an element $A\in \Br(V_2).$ The subset $V_2(K)\subset V_2(L)$ is nonempty. \\		
			For any $v\in S,$ there exist $P_v$ and $Q_v$ in $V_2(K_v)$ such that the local invariants $\inv_v(A(P_v))=0$ and $\inv_v(A(Q_v))=\half.$  For any other $v\notin S,$ and any $P_v\in V_2(K_v),$ the local invariant $\invap=0.$\\		
			For any $v'\in S_L,$ there exist $P_{v'}$ and $Q_{v'}$ in $V_2(L_{v'})$ such that the local invariants $\inv_{v'}(A(P_{v'}))=0$ and $\inv_{v'}(A(Q_{v'}))=\half.$ For any other $v'\notin S_L,$ and any $P_{v'}\in V_2(L_{v'}),$ the local invariant $\inv_{v'}(A(P_{v'})) =0,$ cf. Proposition \ref{proposition the valuation of Brauer group on local points are fixed outside S and take two value on S}.
			\item The Brauer group $\Br(V_3)/\Br(K)\cong\Br(V_{3L})/\Br(L)\cong \ZZ/2\ZZ,$ is generated by an element $A\in \Br(V_3).$ The subset $V_3(\AA_K)\subset V_3(\AA_L)$ is nonempty.\\
			For any $v\in \Omega_K,$ and any $P_v\in V_3(K_v),$ the local invariant $\inv_v(A(P_v))=0$ if $v\notin S;$ the local invariant $\inv_v(A(P_v))=\half$ if $v\in S.$\\
			For any $v'\in \Omega_L,$ and any $P_{v'}\in V_3(L_{v'}),$ the local invariant $\inv_{v'}(A(P_{v'}))=0$ if $v'\notin S_L;$ the local invariant $\inv_{v'}(A(P_{v'}))=\half$ if $v'\in S_L,$ cf. Proposition \ref{proposition the valuation of Brauer group on local points are fixed and on S nontrivial}.
		\end{itemize}
	\end{proposition}
	
	Combining our construction method with the global reciprocity law, we have the following results for Ch\^atelet surfaces. 
	\begin{corollary}[Corollary \ref{interesting result for Chatelet surface1}]
		For any extension of number fields $L/K,$ and any finite nonempty subset  $S \subset \Omega_K\backslash \{$all complex and $2$-adic places$\}$ splitting completely in $L,$ 
		there exists a Ch\^atelet surface $V$ defined over $K$ such that $V(K)\neq \emptyset.$
		For any subfield $L'\subset L$ over $K,$ the Brauer group $\Br(V)/\Br(K)\cong\Br(V_{L'})/\Br(L')\cong \ZZ/2\ZZ.$ And the surface $V_{L'}$ has the following properties.
		\begin{itemize}
			\item For any finite subset $T'\subset \Omega_{L'}$ such that $T'\cap S_{L'}\neq \emptyset,$ the surface $V_{L'}$ satisfies weak approximation off $T'.$ 
			\item  For any finite subset $T'\subset \Omega_{L'}$ such that $T'\cap S_{L'}= \emptyset,$ the surface $V_{L'}$ does not satisfy weak approximation off $T'.$ In particular,  the surface $V_{L'}$ does not satisfy weak approximation.
		\end{itemize}
	\end{corollary}
	
	\begin{corollary}[Corollary \ref{interesting result for Chatelet surface2}]
		For any extension of number fields $L/K,$ there exists a Ch\^atelet surface $V$ defined over $K$ such that $V(\AA_K)\neq \emptyset.$
		For any subfield $L'\subset L$ over $K,$ the Brauer group $\Br(V)/\Br(K)\cong\Br(V_{L'})/\Br(L')\cong \ZZ/2\ZZ.$ And the surface $V_{L'}$ has the following properties.
		\begin{itemize}
			\item If the degree $[L':K]$ is odd, then the surface $V_{L'}$ is a counterexample to the Hasse principle. In particular, the surface $V$ is a counterexample to the Hasse principle.
			\item If the degree $[L':K]$ is even, then the surface $V_{L'}$ satisfies weak approximation. In particular, in this case, the set $V(L')\neq \emptyset.$
		\end{itemize}
	\end{corollary}

	\subsection{Main results for Ch\^atelet surface bundles over curves}	We will apply our results for Ch\^atelet surfaces to construct Ch\^atelet surface bundles over curves to give negative answers to Questions \ref{Questions}.
	
	\subsubsection{A negative answer to Question \ref{question on WA1}} For any quadratic extension of number fields $L/K,$ and assuming the Stoll's conjecture, a Ch\^atelet surface bundle over a curve was constructed by Liang\cite{Li18} to give a negative answer to Question \ref{question on WA1}. Also an unconditional example with explicit equations was given for $L=\QQ(\sqrt{5})$ and $K=\QQ$ in loc. cit. His method only works for quadratic extensions. In this paper, we generalize it to any nontrivial extension of number fields.
	
	For any nontrivial extension of number fields $L/K,$ assuming the Stoll's conjecture, we have the following theorem to give a negative answer to Question \ref{question on WA1}.
	\begin{thm}[Theorem \ref{theorem main result: non-invariance of weak approximation with BMO}]
		For any nontrivial extension of number fields $L/K,$ and any finite subset $T\subset \Omega_L,$ assuming the Stoll's conjecture, there exists a Ch\^atelet surface bundle over a curve: $X\to C$ defined over $K$ such that
		\begin{itemize}
			\item $X$ has a $K$-rational point, and satisfies weak approximation with Brauer-Manin obstruction 
			off $\infty_K,$
			\item $X_L$ does not satisfy weak approximation with Brauer-Manin obstruction 
			off $T.$
		\end{itemize}
	\end{thm}
	
	For $K=\QQ$ and $L=\QQ(\sqrt{3}),$ based on the method given in Theorem \ref{theorem main result: non-invariance of weak approximation with BMO}, we give an explicit unconditional example  in Subsection \ref{subsection main example1}. 
	
	%
	%
	%
	%
	%
	%
	%
	%
	%
	\subsubsection{Negative answers to Question \ref{question on HP}} To the best knowledge of the author, Question \ref{question on HP} has not yet been seriously discussed in the literature.
	
	For any number field $K,$ and any nontrivial field extension $L$ of odd degree over $K,$ assuming the Stoll's conjecture, we have the following theorem to give a negative answer to Question \ref{question on HP}.
	\begin{thm}[Theorem \ref{theorem main result: non-invariance of the Hasse principle with BMO for odd degree}]
		
		For any number field $K,$ and any nontrivial field extension $L$ of odd degree over $K,$ assuming the Stoll's conjecture, there exists a Ch\^atelet surface bundle over a curve: $X\to C$ defined over $K$ such that
		\begin{itemize} 
			\item $X$ is a counterexample to the Hasse principle, and its failure of the Hasse principle is explained by the Brauer-Manin obstruction,
			\item $X_L$ is a counterexample to the Hasse principle, but its failure of the Hasse principle cannot be explained by the Brauer-Manin obstruction.
		\end{itemize}
	\end{thm}
	
	Let $\zeta_7$ be a primitive $7$-th root of unity.
	For $K=\QQ$ and $L=\QQ(\zeta_7+\zeta_7^{-1}),$ based on the method given in Theorem \ref{theorem main result: non-invariance of the Hasse principle with BMO for odd degree}, we give an explicit unconditional example in Subsection \ref{subsection main example2}. The $3$-fold $X$ is a smooth compactification of the following $3$-dimensional affine subvariety given by equations 
	\begin{equation*}
		\begin{cases}
			y^2-377z^2=14(x^4-89726){y'}^2+(x^2-878755181)(5x^2-4393775906)\\ 
			{y'}^2={x'}^3-343x'-2401
		\end{cases}
	\end{equation*}
	in $(x,y,z,x',y')\in \AA^5.$

	For any number field $K$ having a real place, and any nontrivial field extension $L/K$ having a real place, assuming the Stoll's conjecture, we have the following theorem to give a negative answer to Question \ref{question on HP}.
	\begin{thm}[Theorem \ref{theorem main result: non-invariance of the Hasse principle with BMO for exist real place}]
		For any number field $K$ having a real place, and any nontrivial field extension $L/K$ having a real place, assuming the Stoll's conjecture, there exists a Ch\^atelet surface bundle over a curve: $X\to C$ defined over $K$ such that
		\begin{itemize}
			\item $X$ is a counterexample to the Hasse principle, and its failure of the Hasse principle is explained by the Brauer-Manin obstruction,
			\item $X_L$ is a counterexample to the Hasse principle, but its failure of the Hasse principle cannot be explained by the Brauer-Manin obstruction.
		\end{itemize}
	\end{thm}
	
	For $K=\QQ$ and $L=\QQ(\sqrt{3}),$ based on the method given in Theorem \ref{theorem main result: non-invariance of the Hasse principle with BMO for exist real place}, we give an explicit unconditional example in Subsection \ref{subsection main example3}.
	The $3$-fold $X$ is a smooth compactification of the following $3$-dimensional affine subvariety given by equations 
	\begin{equation*}
		\begin{cases}
			y^2+23z^2=5(x^4+805)(x'-4)^2-5(x^4+115)\\
			{y'}^2={x'}^3-16
		\end{cases}
	\end{equation*}
	in $(x,y,z,x',y')\in \AA^5.$

	\begin{exceptions}[Subsection \ref{subsection main example4}]	
		For Question \ref{question on HP}, besides cases of Theorem \ref{theorem main result: non-invariance of the Hasse principle with BMO for odd degree} and Theorem \ref{theorem main result: non-invariance of the Hasse principle with BMO for exist real place}, there are some exceptions. When the degree $[L:K]$ is even and $L$ has no real place, we can give some unconditional examples, case by case, to give negative answers to Question \ref{question on HP}, although we do not have a uniform way to construct them. In Subsection \ref{subsection main example4},  we give an example to explain how it works for
		the case that $K=\QQ$ and $L=\QQ(i).$ The $3$-fold $X,$ is a smooth compactification of the following $3$-dimensional affine subvariety given by equations 
		\begin{equation*}
			\begin{cases}
				y^2+15z^2=(x^4-10x^2+15)({y'}^2+32)/8-(5x^4-39x^2+75)y'/2\\ 
				{y'}^2={x'}^3-16
			\end{cases}
		\end{equation*}
		in $(x,y,z,x',y')\in \AA^5.$
	\end{exceptions}

	\subsubsection{Main ideas behind our constructions in the proof of theorems}
	Given a nontrivial extension of number fields $L/K,$
	we start with a curve $C$ such that $C(K)$ and $C(L)$ are finite, nonempty and $C(K)\neq C(L).$ Using our results for Ch\^atelet surfaces and a construction method from Poonen \cite{Po10}, we construct  a Ch\^atelet surface bundle over this curve: $\beta \colon X\to C$ such that the fiber of each $C(K)$ point is isomorphic to $V_\infty,$ and that the fiber of each $C(L)\backslash C(K)$ point is isomorphic to $V_0.$ Assuming the Stoll's conjecture, by a main proposition from \cite[Proposition 5.4]{Po10} and the functoriality of Brauer-Manin pairing, the Brauer-Manin sets of $X,$ roughly speaking, is the union of adelic points sets of rational fibers. Using some fibration arguments, the arithmetic properties of $V_\infty$ and $V_0$ will determined the arithmetic properties of $X.$ For Theorem \ref{theorem main result: non-invariance of weak approximation with BMO} and \ref{theorem main result: non-invariance of the Hasse principle with BMO for odd degree}, these ideas will be enough. For Theorems \ref{theorem main result: non-invariance of the Hasse principle with BMO for exist real place}, we choose the curve $C$ with an additional connected condition at one real place, and make full use of this real place information.

	\section{Notation and preliminaries}
	Given a number field $K,$ let $\Ocal_K$ be the ring of its integers, and let $\Omega_K$ be the set of all its nontrivial places. Let $\infty_K\subset \Omega_K$ be the subset of all archimedean places, and let $2_K\subset \Omega_K$ be the subset of all $2$-adic places. Let $\infty_K^r\subset \infty_K$ be the subset of all real places, and let $\infty_K^c\subset \infty_K$ be the subset of all complex places. Let $\Omega_K^f=\Omega_K\backslash \infty_K$ be the set of all finite places of $K.$ Let $K_v$ be the completion of $K$ at  $v\in \Omega_K.$ 
	For $v\in \infty_K,$ let $\tau_v\colon K\hookrightarrow K_v$ be the embedding of $K$ into its completion. For $v\in \Omega_K^f,$ let $\Ocal_{K_v}$ be its valuation ring, and let $\FF_v$ be its residue field. 
	Let $S\subset \Omega_K$ be a finite subset, and let $\Ocal_S=\bigcap_{v\in \Omega^f_K\backslash S}(K\cap \Ocal_{K_v})$ be the ring of $S$-integers. Let $\AA_K, \AA_K^S$ be the ring of ad\`eles and ad\`eles without $S$ components of $K.$
	A strong approximation theorem \cite[Chapter II \S 15]{CF67} states that $K$ is dense in $\AA_K^S$ for any nonempty $S.$  
	In this paper, we only use the following special case:
	\begin{lemma}\label{lemma strong approximation for A^1}
		Let $K$ a number field. The set $K$ is dense in $\AA_K^{2_K}.$ 
	\end{lemma}
	
	In this paper, we always assume that a field $L$ is a finite extension of $K.$ Let $S_L\subset \Omega_L$ be the subset of all places above $S.$
	
	It is not difficult to generalize \cite[Theorem 13.4]{Ne99} to the following version of \v{C}ebotarev's density theorem.
	\begin{theorem}[\v{C}ebotarev]\label{theorem Chebotarev density theorem}
		Let $L/K$ be an extension of number fields.
		Then the set of places of $K$ splitting completely in $L,$ has positive density. 
	\end{theorem}
	
	%

	\subsection{Hilbert symbol}

	We use Hilbert symbol $(a,b)_v\in \{\pm 1\},$ for $a,b\in K_v^\times$ and $v\in \Omega_K.$ By definition, $(a,b)_v=1$ if and only if $x_0^2-ax_1^2-bx_2^2=0$ has a $K_v$-solution in $\PP^2$ with homogeneous coordinates $(x_0:x_1:x_2),$ which equivalently means that the curve defined over $K_v$ by the equation $x_0^2-ax_1^2-bx_2^2=0$ in $\PP^2,$ is isomorphic to $\PP^1.$ The Hilbert symbol gives a symmetric bilinear form on $K_v^\times/K_v^{\times 2}$ with value in $\ZZ/2\ZZ,$ cf. \cite[Chapter XIV, Proposition 7]{Se79}. And this bilinear form is nondegenerate, cf. \cite[Chapter XIV, Corollary 7]{Se79}. 
	
	\subsection{Preparation lemmas} We state the following lemmas for later use.

	\begin{lemma}\label{lemma hilbert symbal lifting for odd prime}
		Let $K$ be a number field, and let $v$ be an odd place of $K.$ Let $a, b\in K_v^\times$ such that $v(a),v(b)$ are even. Then $(a,b)_v=1.$
	\end{lemma}
	
	\begin{proof}
		Choose a prime element $\pi_v\in K_v.$ Let $a_1=a\pi_v^{-v(a)}$ and $b_1=b\pi_v^{-v(b)}.$ Since the valuations $v(a)$ and $v(b)$ are even, the elements $\pi_v^{-v(a)}$ and $\pi_v^{-v(b)}$ are in $K_v^{\times 2}.$ So $(a,b)_v=(a_1,b_1)_v$ and $a_1,b_1\in \Ocal_{K_v}^\times.$
		By Chevalley-Warning theorem (cf. \cite[Chapter I \S 2, Corollary 2]{Se73}), the equation $x_0^2-\bar{a}_1x_1^2-\bar{b}_1x_2^2=0$ has a nontrivial solution in $\FF_v.$ For $v$ is odd, by Hensel's lemma, this solution can be lifted to a nontrivial solution in $\Ocal_{K_v}.$ Hence $(a,b)_v=(a_1,b_1)_v=1.$
	\end{proof}
	
	\begin{lemma}\label{lemma Hensel lemma for Hilbert symbal}
		Let $K$ be a number field, and let $v$ be an odd place of $K.$ Let $a,b,c\in K_v^\times$ such that $v(b)<v(c).$ Then $(a,b+c)_v=(a,b)_v.$
	\end{lemma}
	
	\begin{proof}
		For $v(b)<v(c),$ we have $v(b^{-1}c)>0.$ By Hensel's lemma, we have $1+b^{-1}c\in K_v^{\times 2}.$ So $(a,b+c)_v=(a,b(1+b^{-1}c))_v=(a,b)_v.$
	\end{proof}

	\begin{lemma}\label{lemma openness for K^2 and O_K cross}
		Let $K$ be a number field, and let $v\in \Omega_K.$ Then $K_v^{\times 2}$ is an open subgroup of $K_v^\times.$ If $v\in \Omega_K^f,$ then $\Ocal_{K_v}^\times$ is also an open subgroup of $K_v^\times.$ So, they are nonempty open subset of $K_v.$ 
	\end{lemma}
	
	\begin{proof}
		It is obvious for $v\in \infty_K.$ Consider $v\in \Omega_K^f.$
		Let $p$ be the prime number such that $v|p$ in $K.$ Then by Hensel's lemma, the group $K_v^{\times 2}\cap \Ocal_{K_v}^\times$ contains the set $1+ p^3 \Ocal_{K_v},$ which is an open subgroup of $K_v^\times.$   Hence  $K_v^{\times 2}$ and $\Ocal_{K_v}^\times$ are open subgroups of $K_v^\times.$
	\end{proof}

	\begin{lemma}\label{lemma openness for v(x)=n}
		Let $K$ be a number field, and let $v\in \Omega_K^f.$ For any $n\in \ZZ,$ the set $\{x\in K_v|v(x)=n\}$ is a nonempty open subset of $K_v.$
	\end{lemma}
	
	\begin{proof}
		By Lemma \ref{lemma openness for K^2 and O_K cross}, the set $\Ocal_{K_v}^\times$ is an open subgroup of $K_v^\times.$ Choose a prime element $\pi_v\in K_v.$ Then the set $\{x\in K_v^\times|v(x)=n\}=\pi_v^n \Ocal_{K_v}^\times,$ so it is a nonempty open subset of $K_v.$
	\end{proof}

	\begin{lemma}\label{lemma openness for hilbert symbal 1}
		Let $K$ be a number field, and let $v\in \Omega_K^f.$ For any $a\in K_v^\times,$ the sets $\{x\in K_v^\times|(a,x)_v=1\},$ $\{x\in K_v^\times|(a,x)_v=1\}\cap \Ocal_{K_v}$ and $\{x\in \Ocal_{K_v}^\times|(a,x)_v=1\}$ are nonempty open subsets of $K_v.$ 
	\end{lemma}
	
	\begin{proof}
		For the unit $1$ belongs to these sets, they are nonempty. By Lemma \ref{lemma openness for K^2 and O_K cross}, the sets $K_v^{\times 2}$ and $\Ocal_{K_v}^\times$ are nonempty open subsets of $K_v.$ The set $\{x\in K_v^\times|(a,x)_v=1\}$ is a union of cosets of $K_v^{\times 2}$ in the group $K_v^\times.$  So the sets are open in $K_v.$
	\end{proof}

	\begin{lemma}\label{lemma openness for hilbert symbal -1}
		Let $K$ be a number field, and let $v\in \Omega_K^f.$ For any $a\in K_v^\times,$ the sets $\{x\in K_v^\times|(a,x)_v=-1\}$ and  $\{x\in K_v^\times|(a,x)_v=-1\}\cap \Ocal_{K_v}$ are open subsets of $K_v.$ Furthermore, if $a\notin K_v^{\times 2},$ then they are nonempty. 
	\end{lemma}
	
	\begin{proof}
		If the set $\{x\in K_v^\times|(a,x)_v=-1\}\neq \emptyset,$ then it is a union of cosets of $K_v^{\times 2}$ in the group $K_v^\times.$  By Lemma \ref{lemma openness for K^2 and O_K cross}, it is an open subset of $K_v.$ For $\Ocal_{K_v}$ is open in $K_v,$ the sets are open subsets of $K_v^\times.$  Nonemptiness is from the nondegeneracy of the bilinear form given by the Hilbert symbol, and from multiplying a square element in $\Ocal_{K_v}$ to denominate an element in $K_v^\times.$
	\end{proof}

	\begin{lemma}\label{lemma openness for hilbert symbal with odd valuation element}
		Let $K$ be a number field, and let $v\in \Omega_K^f.$ For any $a\in K_v^\times$ with $v(a)$ odd, the set $\{x\in \Ocal_{K_v}^\times|(a,x)_v=-1\}$ is a nonempty open subset of $K_v.$
	\end{lemma}
	
	\begin{proof}
		By Lemmas \ref{lemma openness for K^2 and O_K cross} and \ref{lemma openness for hilbert symbal -1}, the set is open in $K_v.$ We need to show that it is nonempty. For $a\notin K_v^2,$ by the nondegeneracy of the bilinear form given by the Hilbert symbol, there exists an element $b\in K_v^\times$ such that $(a,b)_v=-1.$ If $v(b)$ is odd, let $b'=-ab.$ Then $(a,b')_v=(a,-ab)_v=(a,-a)_v(a,b)_v=-1.$ Replacing $b$ by $b'$ if necessary, we can assume that $v(b)$ is even. Choose a prime element $\pi_v\in K_v.$ Then $\pi_v^{-v(b)}\in K_v^{\times 2},$ so the element $b\pi_v^{-v(b)}$ is in this set.
	\end{proof}

	\begin{lemma}\label{lemma nonsquare element for split place}
		Let $L/K$ be an extension of number fields, and let $v\in \Omega_K\backslash\infty_K^c.$ We assume that $v'\in \Omega_L$ splits over $v,$ i.e. $L_{v'}=K_v.$  Given an element $a\in K,$ if $v$ is a  finite place, we assume that $v(a)$ is odd; if $v$ is an archimedean place, we assume $\tau_v(a)<0.$  Then $a\notin L^{\times 2}.$
	\end{lemma}
	
	\begin{proof}
		The condition that $v(a)$ is odd for the finite place $v,$ or that $\tau_v(a)<0$ for the archimedean place $v,$ implies that $a\notin K_v^{\times 2}.$ For $L_{v'}=K_v,$ we have $a\notin L_{v'}^{\times 2},$ so $a\notin L^{\times 2}.$
	\end{proof}

	\section{Main results for Ch\^atelet surfaces}
	
	In this section, we will construct three kinds of Ch\^atelet surfaces. Each kind in each following subsection has the arithmetic properties mentioned in  Subsection \ref{introduction Main results for Ch\^atelet surfaces}.
	

	Let $K$ be a number field. Ch\^atelet surfaces are smooth projective models of conic bundle surfaces defined by the equation
	\begin{equation}\label{equation}
		y^2-az^2=P(x)
	\end{equation} in $K[x,y,z]$ such that $a\in K^\times,$ and that $P(x)$ is a separable degree-$4$ polynomial in $K[x].$ Given an equation (\ref{equation}), let $V^0$ be the affine surface in $\AA^3_K$ defined by this equation. The natural smooth compactification $V$ of $V^0$ given in \cite[Section 7.1]{Sk01} is called the Ch\^atelet surface given by this equation, cf. \cite[Section 5]{Po09}.
	
	\begin{remark}\label{remark birational to PP^2}
		For any local field $K_v,$  if $a\in K_v^{\times 2},$ then $V$ is birational equivalent to $\PP^2$ over $K_v.$ By the implicit function theorem, there exists a $K_v$-point on $V.$
	\end{remark}
	
	\begin{remark}\label{remark the implicit function thm and local constant Brauer group}
		For any local field $K_v,$ by smoothness of $V,$ the implicit function theorem implies that the nonemptiness of $V^0(K_v)$ is equivalent to the nonemptiness of $V(K_v),$ and that $V^0(K_v)$ is open dense in $V(K_v).$ Given an element $A\in\Br(V),$ the evaluation of $A$ on $V(K_v)$ is locally constant. By the properness of $V,$ the space $V(K_v)$ is compact. So the set of all possible values of the evaluation of $A$ on $V(K_v)$ is finite. Indeed, by \cite[Proposition 7.1.2]{Sk01}, there only exist two possible values. It is determined by the evaluation of $A$ on $V^0(K_v).$ In particular, if the evaluation of $A$ on $V^0(K_v)$ is constant, then it is constant on $V(K_v).$ 
	\end{remark}
	
	\begin{remark}\label{remark independent on the choice of models} 
		In \cite[Proposition 6.1]{CTPS16}, it is shown that whether the Brauer-Manin set of a smooth, projective, and geometrically connected variety is empty, is determined by its birational equivalence class. Here birational equivalence means birational equivalence between smooth, projective, and geometrically connected varieties.
		It is proved in the papers \cite[Theorem B]{CTSSD87a,CTSSD87b} (also explained in the book \cite[Theorem 7.2.1]{Sk01}), that the Brauer-Manin obstruction to the Hasse principle and weak approximation is the only one for Ch\^atelet surfaces.
		Hence, all smooth projective models of a given equation (\ref{equation}) are the same as to the discussion of the Hasse principle, weak approximation, the failure of the Hasse principle explained by the Brauer-Manin obstruction, and weak approximation with Brauer-Manin obstruction.
	\end{remark}

	\begin{remark}\label{remark birational to Hasse-Minkowski theorem}
		If the polynomial $P(x)$ has a factor $x^2-a,$ i.e. there exists a degree-$2$ polynomial $f(x)$ such that $P(x)=f(x)(x^2-a),$ then $Y=\frac{xy+az}{x^2-a}$ and $Z=\frac{y+xz}{x^2-a}$ give a birational equivalence between $V$ and a quadratic surface given by $Y^2-aZ^2=f(x)$ with affine coordinates $(x,Y,Z).$ By the Hasse-Minkowski theorem and Remark \ref{remark independent on the choice of models}, the surface $V$ satisfies weak approximation. 
	\end{remark}

	In this section, we always use the following way to choose an element for the parameter $a$ in the equation (\ref{equation}).
	
	\subsubsection{Choosing an element for the parameter $a$ in the equation (\ref{equation})}\label{subsection choose an element a for S}
	
	Given an extension of number fields $L/K,$
	and a finite subset $S\subset \Omega_K\backslash (\infty_K^c\cup 2_K),$ we will choose an element $a\in \Ocal_K\backslash K^2$ with respect to these $L/K$ and $S$ in the following way.
	
	If $S=\emptyset,$ by Theorem \ref{theorem Chebotarev density theorem}, we can take a place $v_0\in \Omega_K^f\backslash 2_K$ splitting completely in $L.$ Then replace $S$ by $\{v_0\}$ to continue the following step.
	
	Now, suppose that $S\neq \emptyset.$ For $v\in \Omega_K,$ by Lemma \ref{lemma openness for K^2 and O_K cross}, the set $K_v^{\times 2}$ is a nonempty open subset of $K_v.$ For $v\in \Omega_K^f,$ by Lemma \ref{lemma openness for v(x)=n}, the set  $\{a\in K_v| v(a)$ is odd$\}$ is a nonempty open subset of $K_v.$  Using weak approximation for the affine line $\AA^1,$ we can choose an element $a\in K^\times$ satisfying the following conditions:
	\begin{itemize}
		\item $\tau_v(a)<0$ for all $v\in S\cap \infty_K,$
		\item $a\in K_v^{\times 2}$  for all $v\in 2_K,$
		\item $v(a)$ is odd for all $v\in S\backslash \infty_K.$
	\end{itemize}
	These conditions do not change by multiplying an element in $K^{\times 2},$ so we can assume $a\in \Ocal_K.$ The conditions that $v(a)$ is odd for all $v\in S\backslash \infty_K,$ and that $\tau_v(a)<0$ for all $v\in S\cap \infty_K,$ imply $a\in \Ocal_K\backslash K_v^2$ for all $v\in S.$ So $a\in \Ocal_K\backslash K^2.$ 
	
	\begin{remark}\label{remark choose an element a for S remark 1}
		Let $S'=\{v\in \infty_K^r | \tau_v(a)<0\}\cup \{v\in \Omega_K^f\backslash 2_K | v(a){\rm ~is ~odd} \},$ then $S'$ is a finite set. By the conditions that $\tau_v(a)<0$ for all $v\in S\cap \infty_K,$ and that $v(a)$ is odd for all $v\in S\backslash \infty_K,$ we have $S'\supset S.$ Then $S'\neq \emptyset.$
	\end{remark}
	
	\begin{remark}\label{remark choose an element a for S 2}
		If there exists one place in $S$ splitting completely in $L$ or $S=\emptyset,$ then by the choice of $a$ above and Lemma \ref{lemma nonsquare element for split place}, the element $a\in \Ocal_K\backslash L^2.$
	\end{remark}
	
	\begin{remark}\label{remark choose an element a enlarge S 3}
		For the choice of $a,$ we can enlarge $S$ in $\Omega_K\backslash (\infty_K^c\cup 2_K)$ if necessary. 
	\end{remark}

	\subsection{Ch\^atelet surfaces without $K_v$ point for any $v\in S$} In this subsection, we will construct a Ch\^atelet surface of the first kind mentioned in  Subsection \ref{introduction Main results for Ch\^atelet surfaces}.

	\subsubsection{Choice of parameters for the equation (\ref{equation})}\label{subsubsection Choice of parameters for V1}
	Given an extension of number fields $L/K,$
	and a finite subset $S\subset \Omega_K\backslash (\infty_K^c\cup 2_K),$ we choose an element $a\in \Ocal_K\backslash K^2$ as in Subsubsection \ref{subsection choose an element a for S}.
	
	If $S=\emptyset,$ then let $P(x)=1-x^4.$ Then the Ch\^atelet surface $V_1$ given by $y^2-az^2=1-x^4,$ has a rational point $(x,y,z)=(0,1,0).$ 
	
	Now, suppose that $S\neq\emptyset.$ We will choose an element $b\in K^\times$ with respect to the chosen $a$ in the following way.
	
	Let $S'=\{v\in \infty_K^r | \tau_v(a)<0\}\cup \{v\in \Omega_K^f\backslash 2_K | v(a){\rm ~is ~odd} \}$ be as in Remark \ref{remark choose an element a for S remark 1}, then $S'\supset S$ is a finite set. If $v \in S\backslash \infty_K,$ then $v(a)$ is odd, which implies $a\notin K_v^{\times2}.$ Then by Lemma \ref{lemma openness for hilbert symbal -1}, the set  $\{b\in K_v^\times|(a,b)_v=-1\}$ is a nonempty open subset of $K_v.$ If $v \in S'\backslash (S\cup \infty_K),$ then by Lemma \ref{lemma openness for hilbert symbal 1}, the set  $\{b\in K_v^\times|(a,b)_v=1\}$ is a nonempty open subset of $K_v.$ Using weak approximation for affine line $\AA^1,$ we can choose an element $b\in K^\times$ satisfying the following conditions:
	\begin{itemize}
		\item $\tau_v(b)<0$ for all $v\in S\cap \infty_K,$
		\item $\tau_v(b)>0$ for all $v\in (S'\backslash S)\cap \infty_K,$
		\item $(a,b)_v=-1$ for all $v\in S\backslash \infty_K,$
		\item $(a,b)_v=1$ for all $v\in S'\backslash (S\cup \infty_K).$
	\end{itemize}	
	
	We will choose an element $c\in K^\times$ with respect to the chosen $a,b$ in the following way.
	
	Let $S''=\{v\in \Omega_K^f\backslash 2_K| v(b) {\rm ~is ~odd} \},$ then $S''$ is a finite set. The same argument as in the previous paragraph, we can choose an element $c\in K^\times$ satisfying the following conditions:
	\begin{itemize}
		\item $c\in K_v^{\times 2}$ for all $v\in S,$
		\item $v(c)$ is odd for all $v\in S''\backslash S'.$
	\end{itemize}	
	These conditions do not change by multiplying an element in $K^{\times 2},$ so we can assume $b,c\in \Ocal_K.$
	
	Let $P(x)=b(x^4-ac),$ and let $V_1$ be the Ch\^atelet surface given by $y^2-az^2=b(x^4-ac).$

	\begin{proposition}\label{proposition no local point for S}
		For any extension of number fields $L/K,$ and any finite subset $S \subset \Omega_K\backslash (\infty_K^c\cup 2_K)$ splitting completely in $L,$ there exists a Ch\^atelet surface $V_1$ defined over $K$ such that $V_1(\AA_K^S)\subset V_1(\AA_L^{S_L})$ is nonempty, but that $V_1(K_v)=V_1(L_{v'})=\emptyset$  for all $v\in S$ and all $v'\in S_L.$
	\end{proposition}

	\begin{proof}
		For the extension $L/K,$ and the finite set $S,$ we will check that the Ch\^atelet surface $V_1$ chosen as in Subsubsection \ref{subsubsection Choice of parameters for V1}, has the properties.
		
		For $S=\emptyset,$ it is clear.
		
		Now, suppose that $S\neq\emptyset.$ We will check these properties by local computation.
		
		Suppose that $v\in (\infty_K\backslash S')\cup 2_K.$ By the choice of $a,$ we have $a\in K_v^{\times 2}.$ By Remark \ref{remark birational to PP^2}, the surface $V_1$ admits a $K_v$-point.\\
		Suppose that $v\in (S'\backslash S)\cap \infty_K.$ Take $x_0\in K$ such that $\tau_v(x_0^4-ac)>0.$ By the choice of $b,$ we have $\tau_v(b(x_0^4-ac))>0.$ So $(a,b(x_0^4-ac))_v=1,$ which implies that $V_1^0$ admits a $K_v$-point with $x=x_0.$\\
		Suppose that $v\in S'\backslash (S\cup \infty_K),$ then by the choice of $b,$ we have $(a,b)_v=1.$  Take $x_0\in K_v$ such that the valuation $v(x_0)<0,$ then by Lemma \ref{lemma Hensel lemma for Hilbert symbal}, we have $(a,x_0^4-ac)_v=(a,x_0^4)_v=1.$ So $(a,b(x_0^4-ac))_v=(a,b)_v=1,$ which implies that $V_1^0$ admits a $K_v$-point with $x=x_0.$\\
		Suppose that $v\in S''\backslash S'.$ By the choice of $a,b,c,$ the valuations $v(a)$ and $v(-abc)$ are even.  By Lemma \ref{lemma hilbert symbal lifting for odd prime}, we have $(a,-abc)_v=1,$ which implies that $V_1^0$ admits a $K_v$-point with $x=0.$\\
		Suppose that $v\in \Omega_K^f\backslash  (S'\cup S''\cup 2_K).$  Take $x_0\in K_v$ such that the valuation $v(x_0)<0,$ then by Lemma \ref{lemma Hensel lemma for Hilbert symbal}, we have $(a,x_0^4-ac)_v=(a,x_0^4)_v=1.$ For $v(a)$ and $v(b)$ are even, by Lemma \ref{lemma hilbert symbal lifting for odd prime}), we have $(a,b)_v=1.$ So $(a,b(x_0^4-ac))_v=1,$ which implies that $V_1^0$ admits a $K_v$-point with $x=x_0.$\\
		So, the subset $V_1(\AA_K^S)\subset V_1(\AA_L^{S_L})$ is nonempty.
		
		Suppose that $v\in S\cap\infty_K.$ Then by the choice of $a,b,c,$ we have $\tau_v(a), \tau_v(b)$ and $\tau_v(ac)$ are negative. So $(a,b(x^4-ac))_v=-1$ for all $x\in K_v,$  which implies that $V_1^0$ has no $K_v$-point. By Remark \ref{remark the implicit function thm and local constant Brauer group}, we have $V_1(K_v)=\emptyset.$ \\	
		Suppose that $v\in S\backslash\infty_K.$ Then by the choice of $b,$ we have $(a,b)_v=-1.$  Let $x\in K_v.$ If $4v(x)<v(ac),$ then by Lemma \ref{lemma Hensel lemma for Hilbert symbal}, we have $(a,x^4-ac)_v=(a,x^4)_v=1.$ If $4v(x)>v(ac),$ then by Lemma \ref{lemma Hensel lemma for Hilbert symbal}, we have $(a,x^4-ac)_v=(a,-ac)_v=(a,c)_v=1$ (by the choice of $c,$ the last equality holds). For $c\in K_v^{\times 2},$ the valuation $v(ac)$ is odd. So, the equality $4v(x)=v(ac)$ cannot happen. In each case, we have $(a,b(x^4-ac))_v=(a,b)_v(a,x^4-ac)_v=-1,$ which implies that  $V_1^0$ has no $K_v$-point. By Remark \ref{remark the implicit function thm and local constant Brauer group}, we have $V_1(K_v)=\emptyset.$ \\
		So, the set $V_1(K_v)=\emptyset$ for all $v\in S.$	
		
		Take a place $v_0' \in S_L.$ Let $v_0\in S$ be the restriction of $v_0'$ on $K.$ By the assumption that $v_0$ splits completely in $L,$ we have $K_{v_0}=L_{v_0'}.$ Hence $V_1(K_{v_0})=V_1(L_{v_0'}).$ \\
		So, the set $V_1(L_{v'})=\emptyset$  for all $v'\in S_L.$
	\end{proof}
	
	\begin{remark}\label{remark irreducible and reducible of polynomial for no local point}
		By the choice of elements $a,c$ in Subsubsection \ref{subsubsection Choice of parameters for V1}, if $v\in S\backslash\infty_K,$ then
		by comparing the valuation, the polynomial  $P(x)$ is irreducible over $K_v.$ In this case, since the choice of $c$ is not unique, we choose another one to get a new polynomial $P'(X),$ so the polynomials $P(x)$ and $P'(x)$ are coprime in $K.$  For $v\in S\cap\infty_K,$ the polynomial  $P(x)$ is a product of two irreducible factors over $K_v.$  
	\end{remark}

	Using the construction method in Subsubsection \ref{subsubsection Choice of parameters for V1}, we have the following examples, which are special cases of Proposition \ref{proposition no local point for S}. They will be used for further discussion.

	\begin{eg}\label{example2: construction of V_infty}
		Let $K=\QQ,$ and let $\zeta_7$ be a primitive $7$-th root of unity. Let $\alpha=\zeta_7+\zeta_7^{-1}$ with the minimal polynomial $x^3+x^2-2x-1.$ Let  
		$L=\QQ(\alpha).$  Then the degree $[L:K]=3.$ Let $S=\{29\}.$ For $29\equiv 1\mod 7,$ the place $29$ splits completely in $L,$ indeed in $\QQ(\zeta_7).$ Using the construction method in Subsubsection \ref{subsubsection Choice of parameters for V1}, we choose data: $S=\{29\},~S'=\{13,29\},~S''=\{7\},~a=377,~b=14,~c=238$ and  	
		$P(x)=14(x^4-89726).$ Then the Ch\^atelet surface given by $y^2-377z^2=P(x),$
		has the properties of  Proposition \ref{proposition no local point for S}. 
	\end{eg}

	\begin{eg}\label{example3: construction of V_infty}
		Let $K=\QQ,$ and let $L=\QQ(\sqrt{3}).$  Using the construction method in Subsubsection \ref{subsubsection Choice of parameters for V1}, we choose data: $S=\{\infty_K\},~S'=\infty_K\cup \{23\},~S''=\{5\},~a=-23,~b=-5,~c=5$ and  	
		$P(x)=-5(x^4+115).$ Then the Ch\^atelet surface given by $y^2+23z^2=P(x),$
		has the properties of  Proposition \ref{proposition no local point for S}. 
	\end{eg}

	\begin{eg}\label{example3: construction of V_0}
		Let $K=\QQ,$ and let $L=\QQ(\sqrt{3}).$ Then the place $23$ splits completely in $L.$ Using the construction method in Subsubsection \ref{subsubsection Choice of parameters for V1}, we choose data: $S=\{23\},~S'=\infty_K\cup\{23\},~S''=\{5\},~a=-23,~b=5,~c=35$ and  	
		$P(x)=5(x^4+805).$ Then the Ch\^atelet surface given by $y^2+23z^2=P(x),$
		has the properties of  Proposition \ref{proposition no local point for S}. 
	\end{eg}

	\subsection{Ch\^atelet surfaces with rational points and not satisfying weak approximation}

	Given an number field $K,$ Liang \cite[Proposition 3.4]{Li18} constructed a Ch\^atelet surface over $K,$ which has a $K$-rational point and does not satisfy weak approximation off $\infty_K.$ Using the same method as in \cite[Section 5]{Po09} to choose the parameters for the equation (\ref{equation}), he constructed a Ch\^atelet surface, and there exists an element in the Brauer group of this surface, which has two different local invariants on a given finite place, i.e. this element gives an obstruction to weak approximation for this surface. In this subsection, we generalize it.

	Next, we will construct a Ch\^atelet surface of the second kind mentioned in  Subsection \ref{introduction Main results for Ch\^atelet surfaces}.
	
	\subsubsection{Choice of parameters for the equation (\ref{equation})}\label{subsubsection Choice of parameters for V2}
	Given an extension of number fields $L/K,$
	and a finite subset $S\subset \Omega_K\backslash (\infty_K^c\cup 2_K),$ we choose an element $a\in \Ocal_K\backslash K^2$ as in Subsubsection \ref{subsection choose an element a for S}.
	
	We will choose an element $b\in K^\times$ with respect to the chosen $a$ in the following way.
	
	Let $S'=\{v\in \infty_K^r | \tau_v(a)<0\}\cup \{v\in \Omega_K^f\backslash 2_K | v(a){\rm ~is ~odd} \}$ be as in Remark \ref{remark choose an element a for S remark 1}, then $S'\supset S$ is a finite set. 
	By Lemma \ref{lemma openness for v(x)=n}, for $v\in S\backslash \infty_K,$ the set $\{b\in K_v|v(b)=-v(a)\}$ is a nonempty open subset of $K_v;$ for $v\in S'\backslash (S\cup \infty_K),$ the set $\{b\in K_v|v(b)=v(a)\}$ is a nonempty open subset of $\Ocal_{K_v}.$ By Lemma \ref{lemma strong approximation for A^1}, we can choose a nonzero element $b\in \Ocal_S[1/2]$ satisfying the following conditions:
	\begin{itemize}
		\item $v(b)=-v(a)$ for all $v\in S\backslash \infty_K,$
		\item $v(b)=v(a)$  for all $v\in S'\backslash (S\cup \infty_K).$
	\end{itemize}	
	
	We will choose an element $c\in K^\times$ with respect to the chosen $a,b$ in the following way.

	Let $S''=\{v\in \Omega_K^f\backslash 2_K| v(b)\neq 0 \},$ then $S''$ is a finite set and $S'\backslash \infty_K \subset S''.$ 
	By Theorem \ref{theorem Chebotarev density theorem}, we can take two different finite places $v_1,v_2\in \Omega_K^f\backslash  S''$ splitting completely in $L.$ 	
	If $v \in S\backslash \infty_K,$ then $v(a)$ is odd. In this case, by Lemma \ref{lemma openness for hilbert symbal with odd valuation element}, the set  $\{c\in \Ocal_{K_v}^\times|(a,c)_v=-1\}$ is a nonempty open subset of $\Ocal_{K_v}.$ 
	If $v\in \{v_1,v_2\},$ then $b\in \Ocal_{K_v}^\times.$ In this case, by Lemma \ref{lemma openness for v(x)=n}, the sets $\{c\in K_v|v(c)=1\}$ and $\{c\in K_v|v(1+cb^2)=1\}$ are nonempty open subsets of $\Ocal_{K_v}.$ Also by Lemma \ref{lemma strong approximation for A^1}, we can choose a nonzero element $c\in\Ocal_K[1/2]$ satisfying the following conditions:
	\begin{itemize}
		\item $\tau_v(1+cb^2)<0$ for all $v\in S\cap \infty_K,$
		\item $\tau_v(c)>0$ for all $v\in (S'\backslash S)\cap \infty_K,$
		\item $(a,c)_v=-1$ and $v(c)=0$ for all $v\in S\backslash \infty_K,$
		\item $v_1(c)=1$ and $v_2(1+cb^2)=1$ for the chosen $v_1,v_2$ above.
	\end{itemize}
	
	Let $P(x)=(cx^2+1)((1+cb^2)x^2+b^2),$ and let $V_2$ be the Ch\^atelet surface given by $y^2-az^2=(cx^2+1)((1+cb^2)x^2+b^2).$

	\begin{proposition}\label{proposition the valuation of Brauer group on local points are fixed outside S and take two value on S}
		For any extension of number fields $L/K,$ and any finite subset $S \subset \Omega_K\backslash (\infty_K^c\cup 2_K)$ splitting completely in $L,$
		there exists a Ch\^atelet surface $V_2$ defined over $K,$ which has the following properties.
		\begin{itemize}
			\item The Brauer group $\Br(V_2)/\Br(K)\cong\Br(V_{2L})/\Br(L)\cong \ZZ/2\ZZ,$ is generated by an element $A\in \Br(V_2).$ The subset $V_2(K)\subset V_2(L)$ is nonempty. 	
			\item For any $v\in S,$ there exist $P_v$ and $Q_v$ in $V_2(K_v)$ such that the local invariants $\inv_v(A(P_v))=0$ and $\inv_v(A(Q_v))=\half.$  For any other $v\notin S,$ and any $P_v\in V_2(K_v),$ the local invariant $\invap=0.$	
			\item For any $v'\in S_L,$ there exist $P_{v'}$ and $Q_{v'}$ in $V_2(L_{v'})$ such that the local invariants $\inv_{v'}(A(P_{v'}))=0$ and $\inv_{v'}(A(Q_{v'}))=\half.$  For any other $v'\notin S_L,$ and any $P_{v'}\in V_2(L_{v'}),$ the local invariant $\inv_{v'}(A(P_{v'})) =0.$
		\end{itemize}
	\end{proposition}
	
	\begin{proof}
		For the extension $L/K,$ and the finite set $S,$ we will check that the Ch\^atelet surface $V_2$ chosen as in Subsubsection \ref{subsubsection Choice of parameters for V2}, has the properties.

		We need to prove the statement about the Brauer group, and find the element $A$ in this proposition.
		
		By the choice of the places $v_1,$ the polynomial $x^2+c$ is an Eisenstein polynomial, so it is irreducible over $K_{v_1}.$ Since $v_1(a)$ is even, we have $K(\sqrt{a})K_{v_1}\ncong K_{v_1}[x]/(cx^2+1).$ So $K(\sqrt{a})\ncong K[x]/(cx^2+1).$ The same argument holds for the place $v_2$ and polynomial $(1+cb^2)x^2+b^2.$	
		For all places of $S$ split completely in $L,$ then by Remark \ref{remark choose an element a for S 2}, we have $a\in \Ocal_K\backslash L^2.$ 
		By the splitting condition of $v_1, v_2,$ we have $L(\sqrt{a})\ncong L[x]/(cx^2+1)$ and $L(\sqrt{a})\ncong L[x]/((1+cb^2)x^2+b^2).$ So $P(x)=(cx^2+1)((1+cb^2)x^2+b^2)$ is separable and a product of two degree-2 irreducible factors over $K$ and $L.$ 
		According to \cite[Proposition 7.1.1]{Sk01}, the Brauer group $\Br(V_2)/\Br(K)\cong\Br(V_{2L})/\Br(L)\cong \ZZ/2\ZZ.$ Furthermore, by Proposition 7.1.2 in loc. cit, we take the quaternion algebra $A=(a,cx^2+1)\in \Br(V_2)$ as a generator element of this group. Then we have the equality $A=(a,cx^2+1)=(a,(1+cb^2)x^2+b^2)$ in $\Br(V_2).$ 
		
		For $(x,y,z)=(0,b,0)$ is a rational point on $V_2^0,$ the set $V_2(K)$ is nonempty. We denote this rational point by $Q_0.$
		
		We need to compute the evaluation of $A$ on $V_2(K_v)$ for all $v\in \Omega_K.$ 
		
		For any $v\in \Omega_K,$ the local invariant $\invaqo=0.$ By Remark \ref{remark the implicit function thm and local constant Brauer group}, it suffices to compute the local invariant $\invap$ for all $P_v\in V_2^0(K_v).$\\	
		Suppose that $v\in (\infty_K\backslash S')\cup 2_K.$  Then $a\in K_v^{\times 2},$ so $\invap=0$ for all $P_v\in V_2(K_v).$\\
		Suppose that $v\in (S'\backslash S)\cap \infty_K.$ For any $x\in K,$ by the choice of $c,$ we have $\tau_v(cx^2+1)>0.$ Then $(a,cx^2+1)_v=1,$ so $\invap=0$ for all $P_v\in V_2(K_v).$\\
		Suppose that $v\in S'\backslash (S\cup \infty_K).$ Take an arbitrary $P_v\in V_2^0(K_v).$ If $\invap=1/2,$ then $(a,cx^2+1)_v=-1=(a,(1+cb^2)x^2+b^2)_v$ at $P_v.$ By Lemma \ref{lemma Hensel lemma for Hilbert symbal}, the first equality implies $v(x)\leq 0.$ For $v(a)=v(b)>0$ and $v(c)\geq 0,$ by Lemma \ref{lemma Hensel lemma for Hilbert symbal}, we have  $(a,(1+cb^2)x^2+b^2)_v=(a,x^2)_v=1,$ which is a contradiction. So $\invap=0.$\\	
		Suppose that $v\in \Omega_K^f\backslash (S'\cup 2_K ).$ Take an arbitrary $P_v\in V_2^0(K_v).$ If $\invap=1/2,$ then $(a,cx^2+1)_v=-1=(a,(1+cb^2)x^2+b^2)_v$ at $P_v.$ For $v(a)$ is even, by Lemma \ref{lemma hilbert symbal lifting for odd prime}, the first equality implies that $v(cx^2+1)$ is odd. For $c\in\Ocal_K[1/2],$ we have $v(x)\leq 0.$ So $v(c+x^{-2})$ is odd and positive. For $v(b)\geq 0,$ by Hensel's lemma, we have $1+b^2(c+x^{-2})\in K_v^{\times 2}.$ So $(a,(1+cb^2)x^2+b^2)_v=(a,x^2)_v(a,1+b^2(c+x^{-2}))_v=1,$ which is a contradiction. So $\invap=0.$
		
		Suppose that $v\in S\cap\infty_K.$ Take $P_v=Q_0,$ then $\invap=0.$ By the choice of $b,c,$ we have $\tau_v(\frac{b^2}{-cb^2-1})>\tau_v(\frac{1}{-c})>0.$ Take $x_0\in K$ such that $\tau_v(x_0)>\sqrt{\tau_v(\frac{b^2}{-cb^2-1})},$ then $\tau_v((cx_0^2+1)((1+cb^2)x_0^2+b^2))>0$ and $\tau_v(cx_0^2+1)<0.$ So there exists a $Q_v\in V_2^0(K_v)$ with $x=x_0.$ Then $\invaq=\half.$\\	
		Suppose that $v\in S\backslash\infty_K.$ Take $P_v=Q_0,$ then $\invap=0.$ Take $x_0\in K_v$ such that  $v(x_0)< 0.$ For $v(b)=-v(a)<0$ and $v(c)=0,$ by Lemma \ref{lemma Hensel lemma for Hilbert symbal}, we have $(a,cx_0^2+1)_v=(a,cx_0^2)_v=(a,c)_v$ and $(a,(1+cb^2)x_0^2+b^2)_v=(a,cb^2x_0^2)_v=(a,c)_v.$ So
		$(a,(cx_0^2+1)((1+cb^2)x_0^2+b^2))_v=(a,c)_v(a,c)_v=1.$ Hence, there exists a $Q_v\in V_2^0(K_v)$ with $x=x_0.$ For $(a,c)_v=-1,$ we have $\invaq=\half.$ 	
		
		Finally, we need to compute the evaluation of $A$ on $V_2(L_{v'})$ for all $v'\in \Omega_L.$
		
		For any $v'\in \Omega_L,$ the local invariant $\inv_{v'}(A(Q_0))=0.$\\
		Suppose that $v' \in S_L.$ Let $v\in \Omega_K$ be the restriction of $v'$ on $K.$ By the assumption that $v$ is split completely in $L,$ we have $K_v=L_{v'}.$ So $V_2(K_v)=V_2(L_{v'}).$	By the argument already shown, there exist $P_v, Q_v\in V_2(K_v)$ such that $\inv_v(A(P_v))=0$ and $\inv_v(A(Q_v))=\half.$ View $P_v, Q_v$ as elements in $V_2(L_{v'}),$ and let $P_{v'}=P_v$ and $Q_{v'}=Q_v.$ Then $\inv_{v'}(A(P_{v'}))=\invap=0$ and $\inv_{v'}(A(Q_{v'}))=\invaq=\half.$ \\
		Suppose that $v'\in \Omega_L\backslash S_L.$ This local computation is the same as the case $v\in \Omega_K\backslash S.$	
	\end{proof}

	\begin{remark}\label{remark union of the local invariant 0 and one-half}
		For any $v\in S,$ and any $P_v\in V_2(K_v),$  the local invariant of the evaluation of $A$ on $P_v$ is $0$ or $\half.$ Let $U_1=\{P_v\in V_2(K_v)| \invap=0\}$ and $U_2=\{P_v\in V_2(K_v)| \invap=\half\}.$ Then $U_1$ and $U_2$ are nonempty disjoint open subsets of $V_2(K_v),$ and $V_2(K_v)=U_1\bigsqcup U_2.$
	\end{remark}

	
	The following proposition states that the surface $V_2$ in Proposition \ref{proposition the valuation of Brauer group on local points are fixed outside S and take two value on S}, has the following weak approximation properties.
	
	\begin{proposition}\label{proposition the valuation of Brauer group on local points are fixed outside S and take two value on S property}
		Given an extension of number fields $L/K,$
		and a finite subset $S \subset \Omega_K\backslash (\infty_K^c\cup 2_K)$ splitting completely in $L,$ let $V_2$ be a Ch\^atelet surface satisfying those properties of Proposition \ref{proposition the valuation of Brauer group on local points are fixed outside S and take two value on S}. If $S=\emptyset,$ then $V_2$ and $V_{2L}$ satisfy weak approximation.  If $S\neq \emptyset,$ then $V_2$ satisfies weak approximation off $S'$ for any finite subset $S'\subset \Omega_K$ such that $S'\cap S\neq \emptyset,$ while it fails for any finite subset $S'\subset \Omega_K$ such that $S'\cap S= \emptyset.$ And in the case $S\neq \emptyset,$ the surface $V_{2L}$ satisfies weak approximation off $T$ for any finite subset $T\subset \Omega_L$ such that $T\cap S_L\neq \emptyset,$ while it fails for any finite subset $T\subset \Omega_L$ such that $T\cap S_L= \emptyset.$
	\end{proposition}
	
	\begin{proof}
		
		According to \cite[Theorem B]{CTSSD87a,CTSSD87b}, the Brauer-Manin obstruction to the Hasse principle and weak approximation is the only one for Ch\^atelet surfaces, so $V_2(K)$ is dense in $V_2(\AA_K)^{\Br}.$
		
		Suppose that $S=\emptyset,$ then for any $(P_v)_{v\in\Omega_K}\in V_2(\AA_K),$ by Proposition \ref{proposition the valuation of Brauer group on local points are fixed outside S and take two value on S}, the sum $\sum_{v\in \Omega_K}\inv_v(A(P_v))=0.$ For $\Br(V_2)/\Br(K)$ is generated by the element $A,$ we have $V_2(\AA_K)^{\Br}= V_2(\AA_K).$ So $V_2(K)$ is dense in $V_2(\AA_K)^{\Br}= V_2(\AA_K),$ i.e. the surface $V_2$ satisfies weak approximation.  
		
		Suppose that $S'\cap S\neq \emptyset.$ Take $v_0\in S'\cap S.$ 
		For any finite subset $R\subset \Omega_K\backslash\{v_0\},$ take a nonempty open subset $M=V_2(K_{v_0})\times \prod_{v\in R}U_v\times \prod_{v\notin R \cup\{v_0\}}V_2(K_v)\subset V_2(\AA_K).$ Take an element $(P_v)_{v\in \Omega_K}\in M.$ By Proposition \ref{proposition the valuation of Brauer group on local points are fixed outside S and take two value on S} and $v_0\in S,$ we can take an element $P_{v_0}'\in V_2(K_{v_0})$ such that $\inv_{v_0} A(P_{v_0}')=\half.$	
		By Proposition \ref{proposition the valuation of Brauer group on local points are fixed outside S and take two value on S}, the sum $\sum_{v\in \Omega_K\backslash \{v_0\}}\inv_v(A(P_v))$ is $0$ or $\half$ in $\QQ/\ZZ.$ If it is $\half,$ then we replace $P_{v_0}$ by $P_{v_0}'.$ In this way, we get a new element $(P_v)_{v\in \Omega_K}\in M.$  And the sum $\sum_{v\in \Omega_K}\inv_v(A(P_v))=0$ in $\QQ/\ZZ.$ So $(P_v)_{v\in \Omega_K}\in V_2(\AA_K)^{\Br}\cap M.$ For $V_2(K)$ is dense in $V_2(\AA_K)^{\Br},$ the set $V_2(K)\cap M\neq \emptyset,$ which implies that $V_2$ satisfies weak approximation off $\{v_0\}.$ So $V_2$ satisfies weak approximation off $S'.$
		
		Suppose that $S\neq \emptyset$ and $S'\cap S= \emptyset.$ Take $v_0\in S,$ and let $U_{v_0}=\{P_{v_0}\in V_2(K_{v_0})| \inv_{v_0}(A(P_{v_0}))=\half\}.$ For $v\in S\backslash \{v_0\},$  let $U_v=\{P_v\in V_2(K_v)| \invap=0\}.$ For any $v\in S,$ by Remark \ref{remark union of the local invariant 0 and one-half}, the set $U_v$ is a nonempty  open subset of $V_2(K_v).$ Let $M=\prod_{v\in S}U_v\times \prod_{v\notin S }V_2(K_v).$ It is a nonempty  open subset of $V_2(\AA_K).$ For any $(P_v)_{v\in \Omega_K}\in M,$ by Proposition \ref{proposition the valuation of Brauer group on local points are fixed outside S and take two value on S} and the choice of $U_v,$ the sum $\sum_{v\in \Omega_K}\inv_v(A(P_v))=\half$ is nonzero in $\QQ/\ZZ.$ So $V_2(\AA_K)^{\Br}\cap M=\emptyset,$ which implies  $V_2(K)\cap M= \emptyset.$ Hence $V_2$ does not satisfy weak approximation off $S'.$
		
		The same argument applies to $V_{2L}.$ 
	\end{proof}

	Applying the construction method in Subsubsection \ref{subsubsection Choice of parameters for V2}, we have the following weak approximation properties for Ch\^atelet surfaces.
	
	\begin{corollary}\label{interesting result for Chatelet surface1}
		For any extension of number fields $L/K,$ and any finite nonempty subset $S \subset \Omega_K\backslash (\infty_K^c\cup 2_K)$ splitting completely in $L,$
		there exists a Ch\^atelet surface $V$ defined over $K$ such that $V(K)\neq \emptyset.$
		For any subfield $L'\subset L$ over $K,$ the Brauer group $\Br(V)/\Br(K)\cong\Br(V_{L'})/\Br(L')\cong \ZZ/2\ZZ.$ And the surface $V_{L'}$ has the following properties.
		\begin{itemize}
			\item For any finite subset $T'\subset \Omega_{L'}$ such that $T'\cap S_{L'}\neq \emptyset,$ the surface $V_{L'}$ satisfies weak approximation off $T'.$ 
			\item  For any finite subset $T'\subset \Omega_{L'}$ such that $T'\cap S_{L'}= \emptyset,$ the surface $V_{L'}$ does not satisfy weak approximation off $T'.$ In particular,  the surface $V_{L'}$ does not satisfy weak approximation. 
		\end{itemize}
	\end{corollary}

	\begin{proof}
		For the extension $L/K,$ and $S,$ let $V$ be the Ch\^atelet surface chosen as in Subsubsection \ref{subsubsection Choice of parameters for V2}.  Applying the same argument about the field $L$ to its subfield $L',$ the properties that we list, are just what we have explained in Proposition \ref{proposition the valuation of Brauer group on local points are fixed outside S and take two value on S} and Proposition \ref{proposition the valuation of Brauer group on local points are fixed outside S and take two value on S property}.
	\end{proof}

	Using the construction method in Subsubsection \ref{subsubsection Choice of parameters for V2}, we have the following example, which is a special case of Proposition \ref{proposition the valuation of Brauer group on local points are fixed outside S and take two value on S}. It will be used for further discussion.
	
	\begin{eg}\label{example1: construction of V_0}
		Let $K=\QQ,~L=\QQ(\sqrt{3}),$ and let $S=\{73\}.$ The prime numbers $11,23,73$ split completely in $L.$ Using the construction method in Subsubsection \ref{subsubsection Choice of parameters for V2}, we choose data: $S=S'=S''={73},~v_1=11,~v_2=23,~a=73,~b=1/73,~c=99$ and  $P(x)=(99x^2+1)(5428x^2/5329+1/5329).$ Then the Ch\^atelet surface given by $y^2-73z^2=P(x),$ has the properties of  Proposition \ref{proposition the valuation of Brauer group on local points are fixed outside S and take two value on S}, Proposition \ref{proposition the valuation of Brauer group on local points are fixed outside S and take two value on S property}.
	\end{eg}

	\subsection{Ch\^atelet surfaces related to the Hasse principle}

	Iskovskikh \cite{Is71} gave an example of the intersection of two quadratic hypersurfaces in $\PP^4_\QQ,$ which is a Ch\^atelet surface over
	$\QQ$ given by $y^2+z^2=(x^2-2)(-x^2+3).$ He showed that this Ch\^atelet surface is a counterexample to the Hasse principle. Similarly, Skorobogatov \cite[Pages 145-146]{Sk01} gave a family of Ch\^atelet surfaces with a parameter over $\QQ.$ He discussed the property of the Hasse principle for this family. Poonen \cite[Proposition 5.1]{Po09} generalized their arguments to any number field. Given an number field $K,$ he constructed a Ch\^atelet surface defined over $K,$ which is a counterexample to the Hasse principle. He used the \v{C}ebotarev's density theorem for some ray class fields to choose the parameters for the equation (\ref{equation}). The Ch\^atelet surface that he constructed, has the property of \cite[Lemma 5.5]{Po09} (a special situation of the following Proposition \ref{proposition the valuation of Brauer group on local points are fixed and on S nontrivial}: the case when $S=\{v_0\}$ for some place $v_0$ associated to some large prime element in $\Ocal_K$), which is the main ingredient in the proof of \cite[Proposition 5.1]{Po09}. In this subsection, we generalize them.
	
	Next, we will construct a Ch\^atelet surface of the third kind mentioned in  Subsection \ref{introduction Main results for Ch\^atelet surfaces}.

	\subsubsection{Choice of parameters for the equation (\ref{equation})}\label{subsubsection Choice of parameters for V3}
	Given an extension of number fields $L/K,$
	and a finite subset $S\subset \Omega_K\backslash (\infty_K^c\cup 2_K),$ we choose an element $a\in \Ocal_K\backslash K^2$ as in Subsubsection \ref{subsection choose an element a for S}.
	
	We will choose an element $b\in K^\times$ with respect to the chosen $a$ in the following way.
	
	Let $S'=\{v\in \infty_K^r | \tau_v(a)<0\}\cup \{v\in \Omega_K^f\backslash 2_K | v(a){\rm ~is ~odd} \}$ be as in Remark \ref{remark choose an element a for S remark 1}, then $S'\supset S$ is a finite set. If $v \in S\backslash \infty_K,$ then $v(a)$ is odd. Then by Lemma \ref{lemma openness for hilbert symbal with odd valuation element}, the set  $\{b\in \Ocal_{K_v}^\times|(a,b)_v=-1\}$ is a nonempty open subset of $\Ocal_{K_v}.$ If $v \in S'\backslash (S\cup \infty_K),$ then by Lemma \ref{lemma openness for hilbert symbal 1}, the set  $\{b\in \Ocal_{K_v}^\times|(a,b)_v=1\}$ is a nonempty open subset of $\Ocal_{K_v}.$ By Lemma \ref{lemma strong approximation for A^1}, we can choose a nonzero element $b\in \Ocal_K[1/2]$ satisfying the following conditions:
	\begin{itemize}
		\item $\tau_v(b)<0$ for all $v\in S\cap \infty_K,$
		\item $\tau_v(b)>0$ for all $v\in (S'\backslash S)\cap \infty_K,$
		\item $(a,b)_v=-1$ and $v(b)=0$ for all $v\in S\backslash \infty_K,$
		\item $(a,b)_v=1$ and $v(b)=0$  for all $v\in S'\backslash (S\cup \infty_K).$
	\end{itemize}	
	
	We will choose an element $c\in K^\times$ with respect to the chosen $a,b$ in the following way.
	
	Let $S''=\{v\in \Omega_K^f\backslash  2_K | v(b)\neq 0 \},$ then $S''$ is a finite set and $S'\cap S''=\emptyset.$ By Theorem \ref{theorem Chebotarev density theorem}, we can take two different finite places $v_1,v_2\in \Omega_K^f\backslash (S'\cup S''\cup 2_K )$ splitting completely in $L.$ If $v\in (S'\backslash \infty_K)\cup \{v_1,v_2\},$ then $b\in\Ocal_{K_v}^\times.$ In this case, by Lemma \ref{lemma openness for v(x)=n}, the sets $\{c\in K_v|v(bc+1)=v(a)+2\},$ $\{c\in K_v|v(c)=1\}$ and $\{c\in K_v|v(bc+1)=1\}$ are nonempty open subsets of $\Ocal_{K_v}.$
	If $v\in S'',$ by Lemma \ref{lemma openness for hilbert symbal 1}, the set $\{c\in \Ocal_{K_v}^\times|(a,c)_v=1\}$ is a nonempty open subset of $\Ocal_{K_v}.$
	Also by Lemma \ref{lemma strong approximation for A^1}, we can choose a nonzero element $c\in\Ocal_K[1/2]$ satisfying the following conditions:
	\begin{itemize}
		\item $0<\tau_v(c)<-1/\tau_v(b)$ for all $v\in S\cap \infty_K,$
		\item $\tau_v(bc+1)<0$ for all $v\in (S'\backslash S)\cap \infty_K,$
		\item $v(bc+1)=v(a)+2$ for all $v\in S'\backslash \infty_K,$
		\item $(a,c)_v=1$ for all $v\in S'',$
		\item $v_1(c)=1$ and $v_2(bc+1)=1$ for the chosen $v_1,v_2$ above.
	\end{itemize}
	
	
	Let $P(x)=(x^2-c)(bx^2-bc-1),$ and let $V_3$ be the Ch\^atelet surface given by $y^2-az^2=(x^2-c)(bx^2-bc-1).$ 
	
	\begin{proposition}\label{proposition the valuation of Brauer group on local points are fixed and on S nontrivial}
		For any extension of number fields $L/K,$ and any finite subset $S \subset \Omega_K\backslash (\infty_K^c\cup 2_K)$ splitting completely in $L,$ there exists a Ch\^atelet surface $V_3$ defined over $K,$ which has the following properties.
		\begin{itemize}		
			\item The Brauer group $\Br(V_3)/\Br(K)\cong\Br(V_{3L})/\Br(L)\cong \ZZ/2\ZZ,$ is generated by an element $A\in \Br(V_3).$ The subset $V_3(\AA_K)\subset V_3(\AA_L)$ is nonempty.	
			\item For any $v\in \Omega_K,$ and any $P_v\in V_3(K_v),$ 
			\begin{equation*}
				\inv_v(A(P_v))=\begin{cases}
					0& if\quad v\notin S,\\
					1/2 & if \quad v\in S.
				\end{cases}
			\end{equation*}
			\item For any $v'\in \Omega_L,$ and any $P_{v'}\in V_3(L_{v'}),$ 
			\begin{equation*}
				\inv_{v'}(A(P_{v'}))=\begin{cases}
					0& if\quad v'\notin S_L,\\
					1/2 & if \quad v'\in S_L.
				\end{cases}
			\end{equation*}
		\end{itemize}
	\end{proposition}
	
	\begin{proof}
		
		For the extension $L/K,$ and the finite set $S,$ we will check that the Ch\^atelet surface $V_3$ chosen as in Subsubsection \ref{subsubsection Choice of parameters for V3}, has the properties.
		
		Firstly, we need to check that $V_3$ has an $\AA_K$-adelic point.
		
		Suppose that $v\in (\infty_K\backslash S')\cup 2_K.$ Then $a\in K_v^{\times 2}.$ By Remark \ref{remark birational to PP^2}, the surface $V_3$ admits a $K_v$-point.\\	
		Suppose that $v\in (S'\backslash S)\cap \infty_K.$ Let $x_0=0.$ For $\tau_v(b)>0$ and $\tau_v(bc+1)<0,$ we have $\tau_v(c)<0$ and $\tau_v((x_0^2-c)(bx_0^2-bc-1))=\tau_v(c(bc+1))>0,$ which implies that $V_3^0$ admits a $K_v$-point with $x=0.$\\
		Suppose that $v\in S'\backslash (S\cup \infty_K).$  Take $x_0\in K_v$ such that the valuation $v(x_0)<0.$ For $b\in \Ocal_{K_v}^\times$ and $c\in\Ocal_K[1/2],$ by Lemma \ref{lemma Hensel lemma for Hilbert symbal}, we have
		$(a,x_0^2-c)_v=(a,x_0^2)_v=1$ and $(a,bx_0^2-bc-1)_v=(a,bx_0^2)_v=(a,b)_v.$ By the choice of $b,$ we have $(a,b)_v=1.$ Hence $(a,(x_0^2-c)(bx_0^2-bc-1))_v=(a,b)_v=1,$ which implies that $V_3^0$ admits a $K_v$-point with $x=x_0.$\\
		Suppose that $v\in S''.$ By the choice of $a,b,c,$ we have $(a,c)_v=1,$  $v(a)$ even, and $bc+1\in \Ocal_{K_v}^\times.$ By Lemma \ref{lemma hilbert symbal lifting for odd prime}, we have $(a,bc+1)_v=1.$ Let $x_0=0.$ Then $(a,(x_0^2-c)(bx_0^2-bc-1))_v=(a,c(bc+1))_v=(a,c)_v(a,bc+1)_v=1,$ which implies that $V_3^0$ admits a $K_v$-point with $x=0.$\\
		Suppose that $v\in \Omega_K^f\backslash  (S'\cup S''\cup 2_K ).$ Then $v(b)=0.$  Take $x_0\in K_v$ such that the valuation $v(x_0)<0.$ For $b\in \Ocal_{K_v}^\times$ and $c\in\Ocal_K[1/2],$ by Lemma \ref{lemma Hensel lemma for Hilbert symbal}, we have	
		$(a,x_0^2-c)_v=(a,x_0^2)_v=1$ and $(a,bx_0^2-bc-1)_v=(a,bx_0^2)_v=(a,b)_v.$ For $v(a)$ and $v(b)$ are even, by Lemma \ref{lemma hilbert symbal lifting for odd prime}, we have $(a,b)_v=1.$
		So $(a,(x_0^2-c)(bx_0^2-bc-1))_v=(a,b)_v=1,$ which implies that $V_3^0$ admits a $K_v$-point with $x=x_0.$\\
		Suppose that $v\in S\cap\infty_K.$ Let $x_0=0.$ Then by the choice of $a,b,c,$ we have $\tau_v(a)<0,$ $\tau_v(c)>0$ and $\tau_v(bc+1)>0.$ So $(a,(x_0^2-c)(bx_0^2-bc-1))_v=(a,c(bc+1))_v=1,$ which implies that $V_3^0$ admits a $K_v$-point with $x=0.$\\
		Suppose that $v\in S\backslash\infty_K.$ Choose a prime element $\pi_v$ and take $x_0=\pi_v.$ By the choice of $a,b,c,$ we have $b,c\in \Ocal_{K_v}^\times,~v(bx_0^2)=2,$ and $v(bc+1)=v(a)+2\geq 3.$ By Lemma \ref{lemma Hensel lemma for Hilbert symbal}, we have
		$(a,x_0^2-c)_v=(a,-c)_v$ and $(a,bx_0^2-bc-1)_v=(a,bx_0^2)_v.$ By Hensel's lemma, we have $-bc=1-(bc+1)\in K_v^{\times 2}.$
		So $(a,(x_0^2-c)(bx_0^2-bc-1))_v=(a,-bcx_0^2)_v=1,$ which implies that $V_3^0$ admits a $K_v$-point with $x=\pi_v.$
		
		Secondly, we need to prove the statement about the Brauer group, and find the element $A$ in this proposition.
		
		By the choice of the places $v_1,$ the polynomial $x^2-c$ is an Eisenstein polynomial, so it is irreducible over $K_{v_1}.$ Since $v_1(a)$ is even, we have $K(\sqrt{a})K_{v_1}\ncong K_{v_1}[x]/(x^2-c).$ So $K(\sqrt{a})\ncong K[x]/(x^2-c).$ The same argument holds for the place $v_2$ and polynomial $bx^2-bc-1.$
		For all places of $S$ split completely in $L,$ then by Remark \ref{remark choose an element a for S 2}, we have $a\in \Ocal_K\backslash L^2.$ 
		By the splitting condition of $v_1, v_2,$ we have $L(\sqrt{a})\ncong L[x]/(x^2-c)$ and $L(\sqrt{a})\ncong L[x]/(bx^2-bc-1).$
		So $P(x)=(x^2-c)(bx^2-bc-1)$ is separable and a product of two degree-2 irreducible factors over $K$ and $L.$ According to \cite[Proposition 7.1.1]{Sk01}, the Brauer group $\Br(V_3)/\Br(K)\cong\Br(V_{3L})/\Br(L)\cong \ZZ/2\ZZ.$ Furthermore, by Proposition 7.1.2 in loc. cit, we take the quaternion algebra $A=(a,x^2-c)\in \Br(V_3)$ as a generator element of this group. Then we have the equality $A=(a,x^2-c)=(a,bx^2-bc-1)$ in $\Br(V_3).$ 
		
		Thirdly, We need to compute the evaluation of $A$ on $V_3(K_v)$ for all $v\in \Omega_K.$ 
		
		By Remark \ref{remark the implicit function thm and local constant Brauer group}, it suffices to compute the local invariant $\invap$ for all $P_v\in V_3^0(K_v)$ and all $v\in \Omega_K.$
		
		Suppose that $v\in (\infty_K\backslash S')\cup 2_K.$  Then $a\in K_v^{\times 2},$ so $\invap=0$ for all $P_v\in V_3(K_v).$\\
		Suppose that $v\in (S'\backslash S)\cap \infty_K.$  By the choice of $b,c,$ we have $\tau_v(b)>0$ and $\tau_v(bc+1)<0.$ So, for any $x\in K,$ we have $(a,bx^2-bc-1)_v=1.$ Hence $\invap=0$ for all $P_v\in V_3^0(K_v).$\\
		Suppose that $v\in S'\backslash (S\cup \infty_K).$ By the choice of $b,$ we have $(a,b)_v=1.$ Take an arbitrary $P_v\in V_3^0(K_v).$
		If $v(x)<0$ at $P_v,$ by Lemma \ref{lemma Hensel lemma for Hilbert symbal}, we have  $(a,x^2-c)_v=(a,x^2)_v=1.$ 
		If $v(x)> 0$ at $P_v,$ since $b,c\in \Ocal_{K_v}^\times$  and $v(bc+1)=v(a)+2\geq 3,$ by Lemma \ref{lemma Hensel lemma for Hilbert symbal}, we have $(a,x^2-c)_v=(a,-c)_v.$ By Hensel's lemma, we have $-bc=1-(bc+1)\in K_v^{\times 2}.$
		So $(a,x^2-c)_v=(a,-c)_v=(a,-bc)_v=1.$
		If $v(x)=0$ at $P_v,$ since $b\in \Ocal_{K_v}^\times$ and $v(bc+1)=v(a)+2\geq 3,$ by Lemma \ref{lemma Hensel lemma for Hilbert symbal}, we have $(a,bx^2-bc-1)_v=(a,bx^2)_v=1.$ So $\invap=0.$\\ 
		Suppose that $v\in \Omega_K^f\backslash ( S'\cup 2_K ).$ Take an arbitrary $P_v\in V_3^0(K_v).$ If $\invap=1/2,$ then $(a,bx^2-bc-1)_v=(a,x^2-c)_v=-1$ at $P_v.$ For $v(a)$ is even, by Lemma \ref{lemma hilbert symbal lifting for odd prime}, the last equality implies that $v(x^2-c)$ is odd, so it is positive. So $v(bx^2-bc-1)=0.$ By Lemma \ref{lemma hilbert symbal lifting for odd prime}, we have $(a,bx^2-bc-1)_v=1,$ which is a contradiction. So $\invap=0.$
		
		Suppose that $v\in S\cap\infty_K.$  Take an arbitrary $P_v\in V_3^0(K_v).$ If $A(P_v)=0,$ then $(a,bx^2-bc-1)_v=(a,x^2-c)_v=1$ at $P_v.$ The last equality implies that $\tau_v(x^2-c)>0.$  By the choice of $b,$ we have $\tau_v(b)<0,$ so $\tau_v(bx^2-bc-1)<0,$ which contradicts $(a,bx^2-bc-1)_v=1.$ So $\invap=\half.$\\
		Suppose that $v\in S\backslash\infty_K.$ By the choice of $b,$ we have $(a,b)_v=-1.$ Take an arbitrary $P_v\in V_3^0(K_v).$ If $v(x)\leq 0$ at $P_v,$ for $b\in \Ocal_{K_v}^\times$ and $v(bc+1)=v(a)+2\geq 3,$ by Lemma \ref{lemma Hensel lemma for Hilbert symbal}, we have $(a,bx^2-bc-1)_v=(a,bx^2)_v=-1.$ 
		If $v(x)> 0$ at $P_v,$ for $b,c\in \Ocal_{K_v}^\times$ and $v(bc+1)=v(a)+2\geq 3,$ by Lemma \ref{lemma Hensel lemma for Hilbert symbal}, we have $(a,x^2-c)_v=(a,-c)_v.$
		By Hensel's lemma, we have $-bc=1-(bc+1)\in K_v^{\times 2}.$
		So	$(a,x^2-c)_v=(a,-c)_v=-(a,-bc)_v=-1.$ So $\invap=\half.$	
		
		Finally, we need to compute the evaluation of $A$ on $V_3(L_{v'})$ for all $v'\in \Omega_L.$
		
		Suppose that $v' \in S_L.$ Let $v\in \Omega_K$ be the restriction of $v'$ on $K.$ By the assumption that $v$ is split completely in $L,$ we have $K_v=L_{v'}.$ So $V_3(K_v)=V_3(L_{v'}).$ Then for any $P_{v'}\in V_3(L_{v'}),$ denote $P_{v'}$ in $V_3(K_v)$ by $P_v.$ Then by the argument already shown,  the local invariant $\inuap=\invap=\half.$ \\
		Suppose that $v'\in \Omega_L\backslash S_L.$ This local computation is the same as the case $v\in \Omega_K\backslash S.$
	\end{proof}

	\begin{remark}\label{remark the valuation of Brauer group on local points are fixed and on S nontrivial}
		If  the surface $V_3$ has a $K$-rational point $Q,$ then by the global reciprocity law, the sum $\sum_{v\in \Omega_K}\inv_v(A(Q))=0$ in $\QQ/\ZZ.$	
		If the number $\sharp S$ is odd, then from Proposition \ref{proposition the valuation of Brauer group on local points are fixed and on S nontrivial} that we get, this sum is ${\sharp S}/2,$ which is nonzero in $\QQ/\ZZ.$ So, in this case, the surface $V_3$ has no $K$-rational point, which implies that the surface $V_3$ is a counterexample to the Hasse principle.	
		If the number $\sharp S$ is even, then for any $(P_v)_{v\in\Omega_K}\in V_3(\AA_K),$ by Proposition \ref{proposition the valuation of Brauer group on local points are fixed and on S nontrivial}, the sum $\sum_{v\in \Omega_K}\inv_v(A(P_v))={\sharp S}/2=0$ in $\QQ/\ZZ.$ For $\Br(V_3)/\Br(K)$ is generated by the element $A,$ we have $V_3(\AA_K)^{\Br}= V_3(\AA_K)\neq \emptyset.$ According to \cite[Theorem B]{CTSSD87a,CTSSD87b}, the Brauer-Manin obstruction to the Hasse principle and weak approximation is the only one for Ch\^atelet surfaces. So, in this case, the set $V_3(K)\neq \emptyset,$ and it is dense in $V_3(\AA_K)^{\Br}= V_3(\AA_K),$ i.e. the surface $V_3$ has a $K$-rational point and satisfies weak approximation.  In particular, if the number $\sharp S=0,$ i.e. $S=\emptyset,$ though the Brauer group $\Br(V_3)/\Br(K)$ is nontrivial, it gives no obstruction to week approximation for $V_3.$
	\end{remark}

	Applying the construction method in Subsubsection \ref{subsubsection Choice of parameters for V3}, we can relate the properties in Proposition \ref{proposition the valuation of Brauer group on local points are fixed outside S and take two value on S} to the Hasse principle and weak approximation. 
	
	\begin{corollary}\label{interesting result for Chatelet surface2}
		For any extension of number fields $L/K,$ there exists a Ch\^atelet surface $V$ defined over $K$ such that $V(\AA_K)\neq \emptyset.$
		For any subfield $L'\subset L$ over $K,$ the Brauer group $\Br(V)/\Br(K)\cong\Br(V_{L'})/\Br(L')\cong \ZZ/2\ZZ.$ And the surface $V_{L'}$ has the following properties.
		\begin{itemize}
			\item If the degree $[L':K]$ is odd, then the surface $V_{L'}$ is a counterexample to the Hasse principle. In particular, the surface $V$ is a counterexample to the Hasse principle.
			\item If the degree $[L':K]$ is even, then the surface $V_{L'}$ satisfies weak approximation. In particular, in this case, the set $V(L')\neq \emptyset.$
		\end{itemize}
	\end{corollary}

	\begin{proof}
		By Theorem \ref{theorem Chebotarev density theorem}, we can take a place $v_0\in \Omega_K\backslash (\infty_K^c\cup 2_K)$ splitting completely in $L.$ Let $S=\{v_0\}.$ Using the construction method in Subsubsection \ref{subsubsection Choice of parameters for V3}, there exists a Ch\^atelet surface $V$ defined over $K$ having the properties of Proposition \ref{proposition the valuation of Brauer group on local points are fixed and on S nontrivial}. By the same argument as in the proof of Proposition \ref{proposition the valuation of Brauer group on local points are fixed and on S nontrivial}, we have $\Br(V)/\Br(K)\cong\Br(V_{L'})/\Br(L')\cong \ZZ/2\ZZ.$ For $v_0$ splits completely in $L,$ it also does in $L'.$ For $\sharp S$ is odd, if $[L':K]$ is odd, then $\sharp S_{L'}$ is odd; if $[L':K]$ is even, then $\sharp S_{L'}$ is even. Applying the same argument about the field $L$ to its subfield $L',$ the properties that we list, are just what we have explained in Remark \ref{remark the valuation of Brauer group on local points are fixed and on S nontrivial}.
	\end{proof}
	
	\begin{remark}
		Though the Brauer group $\Br(V)/\Br(K)\cong\Br(V_{L'})/\Br(L')\cong \ZZ/2\ZZ$ in Corollary \ref{interesting result for Chatelet surface2}, is nontrivial, it gives an obstruction to the Hasse principle for $V,$ also $V_{L'}$ if $[L':K]$ is odd; but no longer gives an obstruction to week approximation for $V_{L'}$ if $[L':K]$ is even.
	\end{remark}

	Using the construction method in Subsubsection \ref{subsubsection Choice of parameters for V3}, we have the following example, which is a special case of Proposition \ref{proposition the valuation of Brauer group on local points are fixed and on S nontrivial}. It will be used for further discussion.
	
	\begin{eg}\label{example2: construction of V_0}
		Let $K=\QQ$ and $L=\QQ(\zeta_7+\zeta_7^{-1})$ be as in Example \ref{example2: construction of V_infty}, and let $S=\{13\}.$ For $13^2\equiv 1\mod 7,41^2\equiv 1\mod 7$ and $43\equiv 1\mod 7,$ the places $13,41,43$ split completely in $L.$ Using the construction method in Subsubsection \ref{subsubsection Choice of parameters for V3}, we choose data: $S=\{13\},~S'=\{13,29\},~S''=\{5\},~v_1=43,~v_2=41,~a=377,~b=5,~c=878755181$ and  $P(x)=(x^2-878755181)(5x^2-4393775906).$ Then the Ch\^atelet surface given by $y^2-377z^2=P(x),$ has the properties of  Proposition \ref{proposition the valuation of Brauer group on local points are fixed and on S nontrivial}. 
	\end{eg}

	\section{Stoll's conjecture for curves}
	Whether all failures of the Hasse principle of smooth, projective, and geometrically connected curves defined over a number field, are explained by the Brauer-Manin obstruction, was considered by Skorobogatov \cite[Chapter 6.3]{Sk01} and Scharaschkin \cite{Sc99} independently. Furthermore, Stoll \cite[Conjecture 9.1]{St07} made the following conjecture. 
	
	Given a curve $C$ defined over a number field $K,$ let $C(\AA_K)_\bullet=\prod_{v\in \infty_K}\{$connected components of $C(K_v)\}\times C(\AA_K^f).$ The product topology of $\prod_{v\in \infty_K}\{$connected components of $C(K_v)\}$
	with discrete topology and $C(\AA_K^f)$ with adelic topology, gives a topology for $C(\AA_K)_\bullet.$  For any $A\in\Br(C),$ and any $v\in\infty_K,$ the evaluation of $A$ on each connected component of $C(K_v)$ is constant. So, the notation $C(\AA_K)_\bullet^{\Br}$ makes sense.

	\begin{conjecture}\cite[Conjecture 9.1]{St07}\label{conjecture Stoll}
		For any smooth, projective, and geometrically connected curve $C$ defined over a number field $K,$  the set $C(K)$ is dense in $C(\AA_K)_\bullet^{\Br}.$ In particular, the curve $C$ satisfies weak approximation with Brauer-Manin obstruction off $\infty_K.$
	\end{conjecture}

	\begin{remark}
		If Conjecture \ref{conjecture Stoll} holds for a given curve, which is a counterexample to the Hasse principle, then its failure of the Hasse principle is explained by the Brauer-Manin obstruction.  For an elliptic curve defined over $K,$ if its Tate-Shafarevich group is finite, then by the dual sequence of Cassels-Tate, Conjecture \ref{conjecture Stoll} holds for this elliptic curve. 
		With the effort of Kolyvagin \cite{Ko90,Ko91}, Gross and Zagier \cite{GZ86}, and many others, for an elliptic curve defined over $\QQ,$ if its analytic rank equals zero or one, then its Mordell-Weil rank equals its analytic rank, and its Tate-Shafarevich group is finite. So, Conjecture \ref{conjecture Stoll} holds for this elliptic curve.
	\end{remark}

	
	\begin{definition}\label{definition curve of type}
		Given a nontrivial extension of number fields $L/K,$ let $C$ be a smooth, projective, and geometrically connected curve defined over $K.$ We say that a triple $(C,K,L)$  is of type $I$ if $C(K)$ and $C(L)$ are finite nonempty sets, $C(K)\neq C(L)$ and Stoll's Conjecture \ref{conjecture Stoll} holds for the curve $C.$ We say that a triple $(C,K,L)$ is of type $II$ if $(C,K,L)$ is of type $I,$ and there exists a real place $v'\in \infty_L$ such that $C(L_{v'})$ is connected.
	\end{definition}

	\begin{lemma}\label{lemma stoll conjecture}
		Given a nontrivial extension of number fields $L/K,$ if Conjecture \ref{conjecture Stoll} holds for all smooth, projective, and geometrically connected curves defined over $K,$ then there exists a curve $C$ defined over $K$ such that the triple $(C,K,L)$ is of type $I.$ Furthermore, if $L$ has a real place, then this triple $(C,K,L)$ is of type $II.$
	\end{lemma}
	
	\begin{proof}
		Since $L$ is a finite separable extension over $K,$ there exists a $\theta\in L$ such that $L=K(\theta).$ Let $f(x)$ be the minimal polynomial of $\theta.$ Let $n=\deg(f),$ then $n=[L:K]\geq 2.$ Let $\tilde{f}(w_0,w_1)$ be the homogenization of $f.$ If $n$ is odd, we consider a curve $C$ defined over $K$ by a homogeneous equation: $w_2^{n+2}=\tilde{f}(w_0,w_1)(w_1^2-w_0^2)$ with homogeneous coordinates $(w_0:w_1:w_2)\in \PP^2.$ For the polynomials $f(x)$ and $x^2-1$ are separable and coprime in $K[x],$ the curve $C$ is smooth, projective, and geometrically connected. By genus formula for a plane curve, the genus of  $C$ equals $g(C)=n(n+1)/2>1.$ By Faltings's theorem, the sets $C(K)$ and $C(L)$ are finite. It is easy to check that $(w_0:w_1:w_2)=(1:1:0)\in C(K)$ and $(\theta:1:0)\in C(L)\backslash C(K).$ By the assumption that Conjecture \ref{conjecture Stoll} holds for all smooth, projective, and geometrically connected curves over $K,$ we have that the triple $(C,K,L)$ is of type $I.$ Since $n+2$ is odd, the space $C(K_{v'})$ is connected for all $v'\in \infty_L.$ So, if $L$ has a real place, then this triple $(C,K,L)$ is of type $II.$
		If $n$ is even, we replace the homogeneous equation $w_2^{n+2}=\tilde{f}(w_0,w_1)(w_1^2-w_0^2)$ by $w_2^{n+3}=\tilde{f}(w_0,w_1)w_1(w_1^2-w_0^2).$ The same argument applies to this new curve and new triple.
	\end{proof}
	
	\begin{remark}
		For some nonsquare integer $d,$ let $K=\QQ$ and $L=\QQ(\sqrt{d}).$  Consider an elliptic curve $E_d$ defined by a Weierstra\ss~ equation: $y^2=x^3+d.$ 
		Let $E_d^{(d)}$ be the quadratic twist of $E_d$ by $d.$ The curve $E_d$ is connected over $\RR.$ It is easy to check that the point $(x,y)=(0,\sqrt{d})\in C(L)\backslash C(K).$
		If both $E_d(\QQ)$ and $E_d^{(d)}(\QQ)$ are finite, then the set $E_d(L)$ is finite, cf. \cite[Exercise 10.16]{Si09}. If additionally, the Tate-Shafarevich group $\Sha(E_d,\QQ)$ is finite, then the triple $(E_d,K,L)$ is of type $II.$  
	\end{remark}

	\section{Poonen's proposition}

	For our result is base on Poonen's proposition \cite[Proposition 5.4]{Po10}. We recall that paper and his general result first. There exist some remarks on it in \cite[Section 4.1]{Li18}. Colliot-Th\'el\`ene \cite[Proposition 2.1]{CT10} gave another proof of that proposition.

	\begin{recall}\label{recall of Poonen's notation}
		Let $B$ be a smooth, projective, and geometrically connected variety over a number field $K.$ Let $\Lcal$ be a line bundle on $B,$ assuming the set of global sections $\Gamma(B,\Lcal^{\otimes2})\neq 0.$ Let $\Ecal=\Ocal_B\oplus\Ocal_B\oplus \Lcal.$ Let $a$ be a constant in $K^\times,$ and let $s$ be a nonzero global section in $\Gamma(B,\Lcal^{\otimes2}).$ The zero locus of $(1,-a,-s)\in\Gamma(B,\Ocal_B\oplus\Ocal_B\oplus\Lcal^{\otimes2})\subset\Gamma(B,\Sym^2\Ecal)$ in the projective space bundle $\Proj(\Ecal)$ is  a projective and geometrically integral variety, denoted by $X$ with the natural projection $X\to B.$ Let $\overline{K}$ be an algebraic closure field of $K.$ Denote $B\times_{\Spec K} {\Spec \overline{K}}$ by $\overline{B}.$
	\end{recall}	
	\begin{prop}\cite[Proposition 5.3]{Po10}\label{Poonen's main proposition}
		Given a number field $K,$ all notations are the same as in Recall \ref{recall of Poonen's notation}. Let $\alpha\colon X\to B$ be the natural projection.
		Assume that 
		\begin{itemize}
			\item the closed subscheme defined by $s=0$ in $B$ is smooth, projective, and geometrically connected,
			\item $\Br\overline{B}=0$ and $X(\AA_K)\neq \emptyset.$
		\end{itemize}
		Then $X$ is smooth, projective, and geometrically connected. And $\alpha^*\colon\Br(B)\to \Br(X)$ is an isomorphism. 
	\end{prop}

	\section{Main results for Ch\^atelet surface bundles over curves}

	\subsection{Preparation Lemmas} We state the following lemmas, which will be used for the proof of our theorems. 
	
	Fibration methods are used to do research on weak approximation, weak approximation with Brauer-Manin obstruction between two varieties. We modify those fibration methods to fit into our context. 
	\begin{lemma}\label{lemma fiber criterion for wabm}
		Given a number field $K,$ and a finite subset $S\subset \Omega_K,$ let $f\colon X\to Y$ be a $K$-morphism of proper $K$-varieties $X$ and $Y$. We assume that
		\begin{enumerate}{
				\item\label{fiber criterion for wabm condition 1}  the set $Y(K)$ is finite,
				\item\label{fiber criterion for wabm condition 2}  the variety $Y$ satisfies weak approximation with Brauer-Manin obstruction off $S,$
				\item\label{fiber criterion for wabm condition 3}  for any $P\in Y(K),$ the fiber $X_P$ of $f$ over $P$ satisfies weak approximation off $S.$}
		\end{enumerate}
		Then $X$ satisfies weak approximation with Brauer-Manin obstruction off $S.$
	\end{lemma}
	
	\begin{proof}
		For any finite subset $S'\subset\Omega_K\backslash S,$ take an open subset $N=\prod_{v\in S'}U_v\times \prod_{v\notin S'}X(K_v)\subset X(\AA_K)$ such that $N\bigcap X(\AA_K)^{\rm Br}\neq \emptyset.$ Let $M=\prod_{v\in S'}f(U_v)\times \prod_{v\notin S'}f(X(K_v)),$ then by the functoriality of Brauer-Manin pairing, $M\bigcap Y(\AA_K)^{\rm Br}\neq \emptyset.$ By Assumptions (\ref{fiber criterion for wabm condition 1}) and (\ref{fiber criterion for wabm condition 2}), we have $Y(K)= pr^S(Y(\AA_K)^{\rm Br}).$ So there exists $P_0\in pr^S(M)\bigcap Y(K).$ Consider the fiber $X_{P_0}.$ Let $L=\prod_{v\in S'} [X_{P_0}(K_v)\bigcap U_v]\times \prod_{v\notin S'\cup S} X_{P_0}(K_v),$ then it is a nonempty open subset of $X_{P_0}(\AA_K^S).$  By Assumption (\ref{fiber criterion for wabm condition 3}), there exists $Q_0\in L\bigcap X_{P_0}(K).$ So $Q_0\in X(K)\bigcap N,$ which implies that $X$ satisfies weak approximation with Brauer-Manin obstruction off $S.$
	\end{proof}

	\begin{lemma}\label{lemma fiber criterion for not wabm}
		Given a number field $K,$ and a finite subset $S\subset \Omega_K,$ let $f\colon X\to Y$ be a $K$-morphism of proper $K$-varieties $X$ and $Y$. We assume that
		\begin{enumerate}{
				\item\label{fiber criterion for not wabm condition 1} the set $Y(K)$ is finite,
				\item\label{fiber criterion for not wabm condition 2} the morphism $f^*\colon \Br(Y)\to \Br(X)$ is surjective,
				\item\label{fiber criterion for not wabm condition 3} there exists some $P\in Y(K)$ such that the fiber $X_P$ of $f$ over $P$ does not satisfy weak approximation off $S,$ and that $\prod_{v\in S}X_P(K_v)\neq \emptyset.$  }
		\end{enumerate}
		Then $X$ does not satisfy weak approximation with Brauer-Manin obstruction off $S.$
	\end{lemma}
	
	\begin{proof}	
		By Assumption (\ref{fiber criterion for not wabm condition 3}), take a $P_0\in Y(K)$ such that the fiber $X_{P_0}$ does not satisfy weak approximation  off $S,$ and that $\prod_{v\in S} X_{P_0}(K_v)\neq \emptyset.$ Then there exist a finite nonempty subset $S'\subset\Omega_K\backslash S$ and a nonempty open subset $L=\prod_{v\in S'}U_v\times \prod_{v\notin S'}X_{P_0}(K_v)\subset X_{P_0}(\AA_K)$ such that $L\bigcap X_{P_0}(K)=\emptyset.$ 
		By Assumption (\ref{fiber criterion for not wabm condition 1}), the set $Y(K)$ is finite, so we can take a Zariski open subset $V_{P_0}\subset Y$ such that $V_{P_0}(K)=\{P_0\}.$ For any $v\in S',$ since $U_v$ is open in $X_{P_0}(K_v)\subset f^{-1}(V_{P_0})(K_v),$ we can take an open subset $W_v$ of $f^{-1}(V_{P_0})(K_v)$ such that $W_v\cap X_{P_0}(K_v)=U_v.$ 	
		Consider the open subset $N=\prod_{v\in S'}W_v\times \prod_{v\notin S'}X(K_v)\subset X(\AA_K),$ then $L\subset N.$ By the functoriality of Brauer-Manin pairing and Assumption (\ref{fiber criterion for not wabm condition 2}), we have $L\subset N\bigcap X(\AA_K)^{\Br}.$ So $N\bigcap X(\AA_K)^{\Br} \neq \emptyset.$ But $N\bigcap X(K)=N\bigcap X_{P_0}(K)=L\bigcap X_{P_0}(K)
		=\emptyset,$ which implies that $X$ does not satisfy weak approximation with Brauer-Manin obstruction off $S.$
	\end{proof}

	We use the following lemma to choose a dominant morphism from a given curve to $\PP^1.$

	\begin{lemma}\label{lemma choose base change morphism}
		Given a nontrivial extension of number fields $L/K,$ let $C$ be a smooth, projective, and geometrically connected curve defined over $K.$ Assume that the triple $(C,K,L)$ is of type $I$ (Definition \ref{definition curve of type}). For any finite $K$-subscheme $R\subset \PP^1\backslash\{0,\infty\},$  there exists a dominant $K$-morphism $\gamma\colon  C\to \PP^1$ such that $\gamma(C(L)\backslash C(K))=\{0\}\subset \PP^1(K),$ $\gamma(C(K))=\{\infty\}\subset \PP^1(K),$ and that $\gamma$ is \'etale over $R.$ 
	\end{lemma}
	
	\begin{proof}
		Let $K(C)$ be the function field of $C.$ 
		For $C(K)$ and $C(L)$ are finite nonempty sets and $C(L)\backslash C(K)\neq \emptyset,$ by Riemann-Roch theorem, we can choose a rational function $\phi\in K(C)^\times\backslash K^\times$ such that the set of its poles contains $C(K),$ and that the set of its zeros contains $C(L)\backslash C(K).$ 
		This rational function $\phi$ gives a dominant $K$-morphism $\gamma_0\colon C\to \PP^1$ such that $\gamma_0(C(L)\backslash C(K))=\{0\}\subset \PP^1(K)$ and $\gamma_0(C(K))=\{\infty\}\subset \PP^1(K).$  We can choose an automorphism $\varphi_{\lambda_0}\colon \PP^1\to \PP^1, (u:v)\mapsto (\lambda_0 u:v)$ with $\lambda_0\in K^\times$ such that the branch locus of $\gamma_0$ has no intersection with $\varphi_{\lambda_0}(R).$ Let $\lambda= (\varphi_{\lambda_0})^{-1}\circ\gamma_0.$ Then the morphism $\lambda$ is \'etale over $R$ and satisfies other conditions. 
	\end{proof}

	The following lemma is well known.
	
	\begin{lemma}\label{lemma zero of Br B}
		Let $C$ be a curve over a field, and let $B=C\times \PP^1.$ Then $\Br \overline{B}=0.$ 
	\end{lemma}
	
	\begin{proof}
		By \cite[III, Corollary 1.2]{Gr68}, the Brauer group for a given curve over an algebraic closed field is zero. So $\Br(\overline{C}\times \overline{\PP^1})\cong \Br(\overline{C})=0.$
	\end{proof}

	\begin{definition}
		Let $C$ be a smooth, projective, and geometrically connected curve defined over a number field. We say that a morphism
		$\beta\colon X\to C$ is a Ch\^atelet surface bundle over the curve $C,$ if
		\begin{itemize}
			\item $X$ is a smooth, projective, and geometrically connected variety,
			\item the morphism $\beta$ is faithfully flat, and proper,
			\item the generic fiber of $\beta$ is a Ch\^atelet surface over the function field of $C.$
		\end{itemize}
	\end{definition}

	Next, we construct Ch\^atelet surface bundles over curves to give negative answers to Questions \ref{Questions}.

	\subsection{Non-invariance of weak approximation with Brauer-Manin obstruction}
	
	For any quadratic extension of number fields $L/K,$ and assuming Conjecture \ref{conjecture Stoll}, Liang \cite[Theorem 4.5]{Li18} constructed  a Ch\^atelet surface bundle over a curve to give a negative answer to Question \ref{question on WA1}. Assuming that the extension $L/K$ is quadratic, he constructed a Ch\^atelet surface defined over $K$ such that the property of weak approximation is not invariant under the extension of $L/K.$ Then choosing a higher genus curve, he combined this Ch\^atelet surface with the construction method of Poonen \cite{Po10} to get the result. His method only works for quadratic extensions. In this subsection, we generalize his result to any extension $L/K.$

	\begin{theorem}\label{theorem main result: non-invariance of weak approximation with BMO}
		For any nontrivial extension of number fields $L/K,$ and any finite subset $T\subset \Omega_L,$ assuming that Conjecture \ref{conjecture Stoll} holds over $K,$ there exists a Ch\^atelet surface bundle over a curve: $X\to C$ defined over $K$ such that
		\begin{itemize}
			\item $X$ has a $K$-rational point, and satisfies weak approximation with Brauer-Manin obstruction 
			off $\infty_K,$
			\item $X_L$ does not satisfy weak approximation with Brauer-Manin obstruction 
			off $T.$
		\end{itemize}
	\end{theorem}
	
	\begin{proof}
		Firstly, we will construct two Ch\^atelet surfaces.
		Let $S\subset \Omega_K$ be the set of all restrictions of  $T$ on $K.$ By Theorem \ref{theorem Chebotarev density theorem}, we can take a finite place $v_0\in \Omega_K^f\backslash  (S\cup  2_K )$ splitting completely in $L.$ 
		For the extension $L/K,$ and $S=\{v_0\},$ let $V_0$ be the Ch\^atelet surface chosen as in Subsubsection \ref{subsubsection Choice of parameters for V2}. Then $V_0$ defined by $y^2-az^2=P_0(x)$ over $K$ having the properties of Proposition \ref{proposition the valuation of Brauer group on local points are fixed outside S and take two value on S}. 	
		Let $P_\infty(x)=(1-x^2)(x^2-a),$ and let $V_\infty$ be the Ch\^atelet surface defined by $y^2-az^2=P_\infty(x).$ By the argument in the proof of Proposition \ref{proposition the valuation of Brauer group on local points are fixed outside S and take two value on S}, two degree-2 irreducible factors of $P_0(x)$ are prime to $x^2-a$ in $K[x].$ 
		So, the polynomials $P_0(x)$ and $P_\infty(x)$ are coprime in $K[x].$ 
		
		Secondly, we will construct a Ch\^atelet surface bundle over a curve.
		Let $\tilde{P}_\infty(x_0,x_1)$ and $\tilde{P}_0(x_0,x_1)$ be the homogenizations of $P_\infty$ and $P_0.$ Let $(u_0:u_1)\times(x_0:x_1)$ be the coordinates of $\PP^1\times\PP^1,$ and let $s'=u_0^2\tilde{P}_\infty(x_0,x_1)+u_1^2\tilde{P}_0(x_0,x_1)\in \Gamma(\PP^1\times\PP^1,\Ocal(1,2)^{\otimes2}).$ For $P_0(x)$ and $P_\infty(x)$ are coprime in $K[x],$ by Jacobian criterion, the locus $Z'$ defined by $s'=0$ in $\PP^1\times\PP^1$ is smooth. Then the branch locus of the composition $ Z'\hookrightarrow \PP^1\times\PP^1  \stackrel{pr_1}\to\PP^1,$ denoted by $R,$ is finite over $K.$ 
		By the assumption that Conjecture \ref{conjecture Stoll} holds over $K,$ and Lemma \ref{lemma stoll conjecture}, we can take a curve $C$ defined over $K$ such that the triple $(C,K,L)$ is of type $I.$ By Lemma \ref{lemma choose base change morphism}, we can choose a $K$-morphism $\gamma\colon  C\to \PP^1$ such that $\gamma(C(L)\backslash C(K))=\{0\}\subset \PP^1(K),$ $\gamma(C(K))=\{\infty\}\subset \PP^1(K),$ and that $\gamma$ is \'etale over $R.$ 
		Let $B=C\times \PP^1.$ Let $\Lcal=(\gamma,id)^*\Ocal(1,2),$ and let $s=(\gamma,id)^* (s')\in \Gamma(B,\Lcal^{\otimes2}).$ For $\gamma$ is \'etale over the branch locus  $R,$  the locus $Z$ defined by $s=0$ in $B$ is smooth. Since $Z$ is defined by the support of the global section $s,$ it is an effective divisor. The invertible sheaf $\Lscr (Z')$ on $\PP^1\times\PP^1$ is isomorphic to $\Ocal(2,4),$ which is a very ample sheaf on $\PP^1\times\PP^1.$ And $(\gamma,id)$ is a finite morphism, so the pull back of this ample sheaf is again ample, which implies that the invertible sheaf $\Lscr (Z)$ on $C\times\PP^1$ is ample. By \cite[Chapter III Corollary 7.9]{Ha97}, the curve $Z$ is geometrically connected. So the curve $Z$ is smooth, projective, and geometrically connected. By Lemma \ref{lemma zero of Br B}, the Brauer group $\Br(\overline{B})=0.$  Let $X$ be the zero locus of $(1,-a,-s)\in\Gamma(B,\Ocal_B\oplus\Ocal_B\oplus\Lcal^{\otimes2})\subset\Gamma(B,\Sym^2\Ecal)$ in the projective space bundle $\Proj(\Ecal)$ with the natural projection $\alpha\colon  X\to B.$	 Using Proposition \ref{Poonen's main proposition}, the variety $X$ is smooth, projective, and geometrically connected.
		Let $\beta\colon X \stackrel{\alpha}\to B=C\times \PP^1 \stackrel{pr_1}\to C$ be the composition of $\alpha$ and $pr_1.$ Then $\beta$ is a Ch\^atelet surface bundle over the curve $C.$
		
		At last, we will check that $X$ has the properties.
		
		We will show that  $X$ has a $K$-rational point. For any $P\in C(K),$ the fiber  $\beta^{-1}(P)\cong V_\infty.$ The surface $V_\infty$ has a $K$-rational point $(x,y,z)=(0,0,1),$ so the set $X(K)\neq \emptyset.$\\
		We will show that  $X$ satisfies weak approximation with Brauer-Manin obstruction 
		off $\infty_K.$ 
		By Remark \ref{remark birational to Hasse-Minkowski theorem}, the surface $V_\infty$ satisfies weak approximation. So, for the morphism $\beta,$ 
		Assumption (\ref{fiber criterion for wabm condition 3}) of Lemma \ref{lemma fiber criterion for wabm} holds.
		Since Conjecture \ref{conjecture Stoll} holds for the curve $C,$ using Lemma \ref{lemma fiber criterion for wabm} for the morphism $\beta,$ the variety $X$ satisfies weak approximation with Brauer-Manin obstruction 
		off $\infty_K.$ 
		
		We will show that  $X_L$ does not satisfy weak approximation with Brauer-Manin obstruction 
		off $T.$	
		By Proposition \ref{Poonen's main proposition}, the map $\alpha_L^*\colon\Br(B_L)\to \Br(X_L)$ is an isomorphism, so $\beta_L^*\colon\Br(C_L)\to \Br(X_L)$ is an isomorphism. By the choice of the curve $C$ and morphism $\beta,$ for any $Q\in C(L)\backslash C(K),$ the fiber $\beta^{-1}(Q)\cong V_{0L}.$ 
		By Proposition \ref{proposition the valuation of Brauer group on local points are fixed outside S and take two value on S property}, the surface $V_{0L}$ does not satisfy weak approximation off $T\cup \infty_L.$ For $V_0(L)\neq\emptyset,$ by Lemma \ref{lemma fiber criterion for not wabm}, the variety $X_L$ does not satisfy weak approximation with Brauer-Manin obstruction 
		off $T\cup \infty_L.$ So it does not satisfy weak approximation with Brauer-Manin obstruction 
		off $T.$
	\end{proof}

	\subsection{Non-invariance of the failures of the Hasse principle explained by the Brauer-Manin obstruction}

	For extensions $L/K$ of the following two cases, assuming Conjecture \ref{conjecture Stoll}, we construct Ch\^atelet surface bundles over curves to give negative answers to Question \ref{question on HP}.

	\begin{theorem}\label{theorem main result: non-invariance of the Hasse principle with BMO for odd degree}
		For any number field $K,$ and any nontrivial field extension $L$ of odd degree over $K,$ assuming that Conjecture \ref{conjecture Stoll} holds over $K,$ there exists a Ch\^atelet surface bundle over a curve: $X\to C$ defined over $K$ such that
		\begin{itemize} 
			\item $X$ is a counterexample to the Hasse principle, and its failure of the Hasse principle is explained by the Brauer-Manin obstruction,
			\item $X_L$ is a counterexample to the Hasse principle, but its failure of the Hasse principle cannot be explained by the Brauer-Manin obstruction.
		\end{itemize}
	\end{theorem}

	\begin{proof}
		Firstly, we will construct two Ch\^atelet surfaces.
		By Theorem \ref{theorem Chebotarev density theorem}, we can take two different finite places $v_1,v_2\in \Omega_K^f\backslash    2_K $ splitting completely in $L.$ For the extension $L/K,$ and $S=\{v_1,v_2\},$ we choose an element $a$ as in Subsubsection \ref{subsection choose an element a for S}. For the extension $L/K,$ and $S_1=\{v_1\},$ by Remark \ref{remark choose an element a enlarge S 3}, choosing other parameters for the equation (\ref{equation}) as in Subsubsection \ref{subsubsection Choice of parameters for V3}, we have a Ch\^atelet surface $V_0$ defined by $y^2-az^2=P_0(x)$ over $K$ having the properties of Proposition \ref{proposition the valuation of Brauer group on local points are fixed and on S nontrivial}. 
		For the extension $L/K,$ and $S_2=\{v_2\},$ by Remark \ref{remark choose an element a enlarge S 3}, choosing other parameters for the equation (\ref{equation}) as in Subsubsection \ref{subsubsection Choice of parameters for V1}, we have a Ch\^atelet surface $V_\infty$ defined by $y^2-az^2=P_\infty(x)$ over $K$ having the properties of Proposition \ref{proposition no local point for S}.
		By Remark \ref{remark irreducible and reducible of polynomial for no local point}, the polynomial $P_\infty(x)$ is irreducible over $K.$ For $P_0(x)$ is a product of two degree-2 irreducible factors over $K,$ the polynomials $P_0(x)$ and $P_\infty(x)$ are coprime in $K[x].$ 
		
		Secondly, we will construct a Ch\^atelet surface bundle over a curve.
		Let $\tilde{P}_\infty(x_0,x_1)$ and $\tilde{P}_0(x_0,x_1)$ be the homogenizations of $P_\infty$ and $P_0.$ Let $(u_0:u_1)\times(x_0:x_1)$ be the coordinates of $\PP^1\times\PP^1,$ and let $s'=u_0^2\tilde{P}_\infty(x_0,x_1)+u_1^2\tilde{P}_0(x_0,x_1)\in \Gamma(\PP^1\times\PP^1,\Ocal(1,2)^{\otimes2}).$ For $P_0(x)$ and $P_\infty(x)$ are coprime in $K[x],$ by Jacobian criterion, the locus $Z'$ defined by $s'=0$ in $\PP^1\times\PP^1$ is smooth. Then the branch locus of the composition $ Z'\hookrightarrow \PP^1\times\PP^1  \stackrel{pr_1}\to\PP^1,$ denoted by $R,$ is finite over $K.$ 	
		By the assumption that Conjecture \ref{conjecture Stoll} holds over $K,$ and Lemma \ref{lemma stoll conjecture}, we can take a curve $C$ defined over $K$ such that the triple $(C,K,L)$ is of type $I.$ 
		By Lemma \ref{lemma choose base change morphism}, we can choose a $K$-morphism $\gamma\colon  C\to \PP^1$ such that $\gamma(C(L)\backslash C(K))=\{0\}\subset \PP^1(K),$ $\gamma(C(K))=\{\infty\}\subset \PP^1(K),$ and that $\gamma$ is \'etale over $R.$ 
		Let $B=C\times \PP^1.$ Let $\Lcal=(\gamma,id)^*\Ocal(1,2),$ and let $s=(\gamma,id)^* (s')\in \Gamma(B,\Lcal^{\otimes2}).$ By the same argument as in the proof of Theorem \ref{theorem main result: non-invariance of weak approximation with BMO},	
		the locus $Z$ defined by $s=0$ in $B$ is smooth, projective, and geometrically connected; the Brauer group $\Br(\overline{B})=0.$ Let $X$ be the zero locus of $(1,-a,-s)\in\Gamma(B,\Ocal_B\oplus\Ocal_B\oplus\Lcal^{\otimes2})\subset\Gamma(B,\Sym^2\Ecal)$ in the projective space bundle $\Proj(\Ecal)$ with the natural projection $\alpha\colon  X\to B.$ By Proposition \ref{Poonen's main proposition}, the variety $X$ is smooth, projective, and geometrically connected.
		Let $ \beta\colon X \to C$ be the composition of $\alpha$ and $pr_1.$ Then $\beta$ is a Ch\^atelet surface bundle over the curve $C.$
		
		At last, we will check that $X$ has the properties.

		We will show $X(\AA_K)\neq \emptyset.$
		For any $P\in C(K),$ the fiber  $\beta^{-1}(P)\cong V_\infty.$ By Proposition \ref{proposition no local point for S}, the set $V_\infty(\AA_K^{\{v_2\}})\neq \emptyset.$ So $X(\AA_K^{\{v_2\}})\neq \emptyset.$ For $v_2$ splits completely in $L,$ take a place $v_2'\in \Omega_L^f$ above $v_2,$ i.e. $v_2'|v_2$ in $L.$ Then $K_{v_2}=L_{v_2'}.$ By Proposition \ref{proposition the valuation of Brauer group on local points are fixed and on S nontrivial}, the set $V_0(\AA_L)\neq \emptyset.$ Take a point $Q\in C(L)\backslash C(K),$ then the fiber $\beta^{-1}(Q)\cong V_{0L}.$	
		We have  $X(K_{v_2})=X_L(L_{v_2'})\supset \beta^{-1}(Q)(L_{v_2'})\cong V_0((L_{v_2'})\neq \emptyset.$ So $X(\AA_K)\neq \emptyset.$\\
		We will show $X(\AA_K)^{\Br}=\emptyset.$	 By Conjecture \ref{conjecture Stoll}, the set $C(K)$ is finite, and $C(K)=pr^{\infty_K}(C(\AA_K)^{\Br}).$ By the functoriality of Brauer-Manin
		pairing, we have  $pr^{\infty_K}(X(\AA_K)^{\Br})\subset  \bigsqcup_{P\in C(K)}\beta^{-1}(P)(\AA_K^{\infty_K}).$ But by Proposition \ref{proposition no local point for S}, the set $V_\infty(K_{v_2})=\emptyset,$ so we have
		$pr^{\infty_K}(X(\AA_K)^{\Br})\subset \bigsqcup_{P\in C(K)}\beta^{-1}(P)(\AA_K^{\infty_K})\cong V_\infty(\AA_K^{\infty_K})\times C(K)=\emptyset,$ which implies that $ X(\AA_K)^{\Br}=\emptyset.$ \\
		So, the variety $X$ is a counterexample to the Hasse principle, and its failure of the Hasse principle is explained by the Brauer-Manin obstruction.
		
		We will show $X_L(\AA_L)^{\Br}\neq \emptyset.$
		By Proposition \ref{Poonen's main proposition}, the map $\alpha_L^*\colon\Br(B_L)\to \Br(X_L)$ is an isomorphism, so $\beta_L^*\colon\Br(C_L)\to \Br(X_L)$ is an isomorphism. By the functoriality of Brauer-Manin pairing, the set  $X_L(\AA_L)^{\Br}$ contains
		$\bigsqcup_{Q\in C(L)\backslash C(K) }\beta^{-1}(Q)(\AA_L)\cong V_0(\AA_L)\times (C(L)\backslash C(K)),$ which is nonempty.\\
		We will show $X(L)=\emptyset.$ By the assumption that the degree $[L:K]$ is odd, and $v_1$ splitting completely in $L,$  the number $\sharp S_{1L}$ is odd. By Proposition \ref{proposition the valuation of Brauer group on local points are fixed and on S nontrivial} and the global reciprocity law explained in Remark \ref{remark the valuation of Brauer group on local points are fixed and on S nontrivial}, the set $V_0(L)=\emptyset.$ By Proposition \ref{proposition no local point for S}, the set $V_\infty(\AA_L)=\emptyset.$ Since each $L$-rational fiber of $\beta$ is isomorphic to $V_{0L}$ or $V_{\infty L},$ the set $X(L)=\emptyset.$	\\
		So,
		the variety $X_L$ is a counterexample to the Hasse principle, but its failure of the Hasse principle cannot be explained by the Brauer-Manin obstruction.
	\end{proof}

	In the case when both fields $K$ and $L$ have real places, making use of these real place information, and assuming the Stoll's conjecture, we have the following theorem to give a negative answer to Question \ref{question on HP}

	\begin{theorem}\label{theorem main result: non-invariance of the Hasse principle with BMO for exist real place}
		For any number field $K$ having a real place, and any nontrivial field extension $L/K$ having a real place, assuming that Conjecture \ref{conjecture Stoll} holds over $K,$  there exists a Ch\^atelet surface bundle over a curve: $X\to C$ defined over $K$ such that
		\begin{itemize}
			\item $X$ is a counterexample to the Hasse principle, and its failure of the Hasse principle is explained by the Brauer-Manin obstruction,
			\item $X_L$ is a counterexample to the Hasse principle, but its failure of the Hasse principle cannot be explained by the Brauer-Manin obstruction.
		\end{itemize}
	\end{theorem}

	\begin{proof}
		Firstly, we will construct two Ch\^atelet surfaces.
		Take a real place $v_0'$ of $L,$ and let $v_0\in \infty_K$ be the restriction of  $v_0'$ on $K.$ 
		By Theorem \ref{theorem Chebotarev density theorem}, we can take a finite place $v_1\in \Omega_K^f\backslash    2_K $ splitting completely in $L.$ 	
		For the extension $L/K,$ and $S=\{v_0,v_1\},$ we choose an element $a$ as in Subsubsection \ref{subsection choose an element a for S}. For the trivial extension $K/K$ (respectively the extension $L/K$), and $S_1=\{v_0\}$ (respectively $S_2=\{v_1\}$), by Remark \ref{remark choose an element a enlarge S 3}, choosing other parameters for the equation (\ref{equation}) as in Subsubsection \ref{subsubsection Choice of parameters for V1}, we have a Ch\^atelet surface $V_0$ (respectively $V_\infty$) defined by $y^2-az^2=P_0(x)$ (respectively by $y^2-az^2=P_\infty(x)$) over $K$ having the properties of Proposition \ref{proposition no local point for S}.
		By Remark \ref{remark irreducible and reducible of polynomial for no local point},	
		the polynomial $P_\infty(x)$ is irreducible over $K_{v_1}.$ If the polynomials
		$P_0(x)$ and $P_\infty(x)$ are not coprime, then $P_0(x)=\lambda P_\infty(x)$ for some $\lambda\in K^\times.$ By Remark \ref{remark irreducible and reducible of polynomial for no local point}, we choose another $P_\infty(x)'$ to replace $P_\infty(x)$ so that the new polynomial $P_\infty(x)$ is prime to  $P_0(x).$  
		
		Secondly, we will construct a Ch\^atelet surface bundle over a curve.
		Let $\tilde{P}_\infty(x_0,x_1)$ and $\tilde{P}_0(x_0,x_1)$ be the homogenizations of $P_\infty$ and $P_0.$ Let $(u_0:u_1)\times(x_0:x_1)$ be the coordinates of $\PP^1\times\PP^1,$ and let $s'=u_0^2\tilde{P}_\infty(x_0,x_1)+u_1^2\tilde{P}_0(x_0,x_1)\in \Gamma(\PP^1\times\PP^1,\Ocal(1,2)^{\otimes2}).$ For $P_0(x)$ and $P_\infty(x)$ are coprime in $K[x],$ by Jacobian criterion, the locus $Z'$ defined by $s'=0$ in $\PP^1\times\PP^1$ is smooth. Then the branch locus of the composition $ Z'\hookrightarrow \PP^1\times\PP^1  \stackrel{pr_1}\to\PP^1,$ denoted by $R,$ is finite over $K.$ 
		By the assumptions that Conjecture \ref{conjecture Stoll} holds over $K,$ and that the field $L$ has a real place, using Lemma \ref{lemma stoll conjecture}, we can take a curve $C$ defined over $K$ such that the triple $(C,K,L)$ is of type $II.$ 
		By Lemma \ref{lemma choose base change morphism}, we can choose a $K$-morphism $\gamma\colon  C\to \PP^1$ such that $\gamma(C(L)\backslash C(K))=\{0\}\subset \PP^1(K),$ $\gamma(C(K))=\{\infty\}\subset \PP^1(K),$ and that $\gamma$ is \'etale over $R.$ 
		Let $B=C\times \PP^1.$ Let $\Lcal=(\gamma,id)^*\Ocal(1,2),$ and let $s=(\gamma,id)^* (s')\in \Gamma(B,\Lcal^{\otimes2}).$ By the same argument as in the proof of Theorem \ref{theorem main result: non-invariance of weak approximation with BMO},	
		the locus $Z$ defined by $s=0$ in $B$ is smooth, projective, and geometrically connected; the Brauer group $\Br(\overline{B})=0.$  Let $X$ be the zero locus of $(1,-a,-s)\in\Gamma(B,\Ocal_B\oplus\Ocal_B\oplus\Lcal^{\otimes2})\subset\Gamma(B,\Sym^2\Ecal)$ in the projective space bundle $\Proj(\Ecal)$ with the natural projection $\alpha\colon  X\to B.$  By the same argument as in the proof of Theorem \ref{theorem main result: non-invariance of the Hasse principle with BMO for odd degree}, the variety $X$ is smooth, projective, and geometrically connected.
		Let $ \beta\colon X \to C$ be the composition of $\alpha$ and $pr_1.$ Then $\beta$ is a Ch\^atelet surface bundle over the curve $C.$
		
		At last, we will check that $X$ has the properties.
		
		We will show $X(\AA_K)\neq \emptyset.$ For any $P\in C(K),$ the fiber  $\beta^{-1}(P)\cong V_\infty.$ By Proposition \ref{proposition no local point for S}, 
		the set $V_\infty(\AA_K^{\{v_1\}})\neq \emptyset.$ 
		So $X(\AA_K^{\{v_1\}})\neq \emptyset.$ For $v_1$ splits completely in $L,$ take a place $v_1'\in \Omega_L^f$ above $v_1,$ i.e. $v_1'|v_1$ in $L.$ Then $K_{v_1}=L_{v_1'}.$ By Proposition \ref{proposition no local point for S}, the set $V_0(K_{v_1})\neq \emptyset.$ 
		Take a point $Q\in C(L)\backslash C(K),$ then the fiber $\beta^{-1}(Q)\cong V_{0L}.$	
		We have  $X(K_{v_1})=X_L(L_{v_1'})\supset \beta^{-1}(Q)(L_{v_1'})\cong V_0(L_{v_1'})\neq \emptyset.$ 
		So $X(\AA_K)\neq \emptyset.$\\
		By the same argument as in the proof of Theorem \ref{theorem main result: non-invariance of the Hasse principle with BMO for odd degree}, the set $X(\AA_K)^{\Br}=\emptyset.$ So, the variety $X$ is a counterexample to the Hasse principle, and its failure of the Hasse principle is explained by the Brauer-Manin obstruction.
		

		We will show $X_L(\AA_L)^{\Br}\neq \emptyset.$
		By Proposition \ref{proposition no local point for S}, the set $V_0(\AA_K^{\{v_0\}})\neq \emptyset$ and $V_\infty(K_{v_0})\neq \emptyset.$  By Proposition \ref{Poonen's main proposition}, the map $\beta_L^*\colon\Br(C_L)\to \Br(X_L)$ is an isomorphism. By our choice, the space $C(K_{v_0})\cong C(L_{v_0'})\cong C(\RR)$ is connected.  For any $A\in\Br(C_L),$ since the evaluation of $A$ on $C(L_{v'})$ is locally constant for all $v' \in \Omega_L,$ it is constant on $C(K_{v_0}),$ so it is constant on $C(L_{v'})$ for all $v'\in S_{1L}.$	Take points $P\in C(K)$ and $Q\in C(L)\backslash C(K).$
		By the functoriality of Brauer-Manin pairing and isomorphism of $\beta_L^*\colon\Br(C_L)\to \Br(X_L),$ the set  $X_L(\AA_L)^{\Br}\supset 
		\beta^{-1}(Q)(\AA_L^{S_{1L}})\times \prod_{v'\in S_{1L}}\beta^{-1}(P)(L_{v'})\cong V_0(\AA_L^{S_{1L}})\times \prod_{v'\in S_{1L}}V_\infty(L_{v'})\neq \emptyset.$\\
		We will show $X(L)=\emptyset.$ By Proposition \ref{proposition no local point for S}, the set $V_0(K_{v_0})=\emptyset,$ so the set $V_0(L_{v_0'})=V_0(K_{v_0})=\emptyset.$ By Proposition \ref{proposition no local point for S}, 
		the set $V_\infty(\AA_L)=\emptyset.$ Since each $L$-rational fiber of $\beta$ is isomorphic to $V_{0L}$ or $V_{\infty L},$ the set $X(L)=\emptyset.$\\
		So, 
		the variety $X$ is a counterexample to the Hasse principle, but its failure of the Hasse principle cannot be explained by the Brauer-Manin obstruction.
	\end{proof}


	\section{Explicit unconditional examples}

	Firstly, we will give explicit examples without assuming Conjecture \ref{conjecture Stoll} for Theorem \ref{theorem main result: non-invariance of weak approximation with BMO}, Theorem \ref{theorem main result: non-invariance of the Hasse principle with BMO for odd degree} and Theorem \ref{theorem main result: non-invariance of the Hasse principle with BMO for exist real place}. Secondly, when $K=\QQ$ and $L=\QQ(i),$ besides cases of Theorem \ref{theorem main result: non-invariance of the Hasse principle with BMO for odd degree} and Theorem \ref{theorem main result: non-invariance of the Hasse principle with BMO for exist real place}, we construct an explicit Ch\^atelet surface bundle over a curve in Subsection \ref{subsection main example4} to give a negative answer to Question \ref{question on HP}.

	\subsection{An explicit unconditional example for Theorem \ref{theorem main result: non-invariance of weak approximation with BMO}}\label{subsection main example1}
	
	In the subsection,
	let $K=\QQ$ and $L=\QQ(\sqrt{3}).$ We will construct an explicit Ch\^atelet surface bundle over a curve having properties of Theorem \ref{theorem main result: non-invariance of weak approximation with BMO}.

	\subsubsection{Choosing an elliptic curve} 
	Let $E$ be an elliptic curve defined over $\QQ$ by a homogeneous equation:
	$$w_1^2w_2=w_0^3-16w_2^3$$ with homogeneous coordinates $(w_0:w_1:w_2)\in \PP^2.$ This is an elliptic curve with complex multiplication. Its quadratic twist $E^{(3)}$ is isomorphic to an elliptic curve defined by a homogeneous equation:
	$w_1^2w_2=w_0^3-432w_2^3$ with homogeneous coordinates $(w_0:w_1:w_2)\in \PP^2.$ These elliptic curves $E$ and $E^{(3)}$ defined over $\QQ,$ are  of analytic rank $0.$ Then the Tate-Shafarevich group $\Sha(E,\QQ)$ is finite, so $E$ satisfies weak approximation with Brauer-Manin obstruction off $\infty_K.$ The Mordell-Weil groups $E(K)$ and $E^{(3)}(K)$ are finite, so $E(L)$ is finite. Indeed, the Mordell-Weil groups $E(K)=\{(0:1:0)\}$ and $E(L)=\{(4:\pm 4\sqrt{3}:1),(0:1:0)\}.$ So the triple $(E,K,L)$ is of type $I.$


	\subsubsection{Choosing a dominant morphism}
	Let $\PP^2\backslash\{(0:1:0)\}\to \PP^1$ be a morphism over $\QQ$ given by $(w_0:w_1:w_2)\mapsto(w_0-4w_2:w_2).$ Composite with the natural inclusion $E\backslash\{(0:1:0)\}\hookrightarrow \PP^2\backslash\{(0:1:0)\}.$ We get a morphism $E\backslash\{(0:1:0)\}\to \PP^1,$ which can be extended to a dominant morphism $\gamma\colon E\to \PP^1$  of degree $2.$ The morphism $\gamma$ maps $E(K)$ to $\{\infty\}=\{(1:0)\},$ and maps $(4:\pm 4\sqrt{3}:1)$ to $0$ point: $(0:1).$	By B\'ezout's Theorem \cite[Chapter I. Corollary 7.8]{Ha97} or Hurwitz's Theorem \cite[Chapter IV. Corollary 2.4]{Ha97}, the branch locus of $\gamma$ is $\{(1:0),(2\sqrt[3]{2}-4:1),(2\sqrt[3]{2}e^{2\pi i/3}-4:1),(2\sqrt[3]{2}e^{-2\pi i/3}-4:1)\}.$

	\subsubsection{Construction of a Ch\^atelet surface bundle}

	Let $P_\infty(x)=(1-x^2)(x^2-73),$ and let $P_0(x)=(99x^2+1)(5428x^2/5329+1/5329).$ Notice that these polynomials $P_\infty$ and $P_0$ are separable.
	Let $V_\infty$ be the Ch\^atelet surface given by $y^2-73z^2=P_\infty(x).$ As mentioned in Example \ref{example1: construction of V_0}, let $V_0$ be the Ch\^atelet surface given by $y^2-73z^2=P_0(x).$ 
	Let $\tilde{P}_\infty(x_0,x_1)$ and $\tilde{P}_0(x_0,x_1)$ be the homogenizations of $P_\infty$ and $P_0.$ Let $(u_0:u_1)\times(x_0:x_1)$ be the coordinates of $\PP^1\times\PP^1,$ and let $s'=u_0^2\tilde{P}_\infty(x_0,x_1)+u_1^2\tilde{P}_0(x_0,x_1)\in \Gamma(\PP^1\times\PP^1,\Ocal(1,2)^{\otimes2}).$ For $P_0(x)$ and $P_\infty(x)$ are coprime in $K[x],$ by Jacobian criterion, the locus $Z'$ defined by $s'=0$ in $\PP^1\times\PP^1$ is smooth. Then the branch locus of the composition $ Z'\hookrightarrow \PP^1\times\PP^1  \stackrel{pr_1}\to\PP^1,$ denoted by $R,$ is finite, and contained in $\PP^1\backslash \{(1:0)\}.$
	Let $B=E\times \PP^1.$ Let $\Lcal=(\gamma,id)^*\Ocal(1,2),$ and let $s=(\gamma,id)^* (s')\in \Gamma(B,\Lcal^{\otimes2}).$ With these notations, we have the following lemma.

	\begin{lemma}
		The curve $Z$ defined by $s=0$ in $B$ is smooth, projective, and geometrically connected.
	\end{lemma}
	
	\begin{proof}
		For smoothness of $Z,$ we need to check that the branch locus $R$ 
		does not intersect with the branch locus of  $\gamma\colon E\to \PP^1$. 
		For $R$ is contained in $\PP^1\backslash \{(1:0)\},$ 
		we can assume the homogeneous coordinate $u_1=1,$ then the point in $R$ satisfies one of the following equations: $
		5329u_0^2-537372=0,~ 389017u_0^2 - 1=0,~27625536u_0^4 +157730624u_0^2 + 5329=0.$ The polynomials of these equations are irreducible over $\QQ.$  By comparing the degree $[\QQ(u_0):\QQ]$ with the branch locus of $\gamma,$ we get the conclusion that these two branch loci do not intersect.
		The same argument as in the proof of Theorem \ref{theorem main result: non-invariance of weak approximation with BMO},	
		the locus $Z$ defined by $s=0$ in $B$ is geometrically connected. So it is smooth, projective, and geometrically connected.
	\end{proof}

	Let $X$ be the zero locus of $(1,-a,-s)\in\Gamma(B,\Ocal_B\oplus\Ocal_B\oplus\Lcal^{\otimes2})\subset\Gamma(B,\Sym^2\Ecal)$ in the projective space bundle $\Proj(\Ecal)$ with the natural projection $\alpha\colon  X\to B.$  By the same argument as in the proof of Theorem \ref{theorem main result: non-invariance of the Hasse principle with BMO for odd degree}, the variety $X$ is smooth, projective, and geometrically connected. Let $ \beta\colon X \to E$ be the composition of $\alpha$ and $pr_1.$ Then it is a Ch\^atelet surface bundle over the curve $E.$ For this $X,$ we have the following proposition.

	\begin{prop}\label{example1: main result: non-invariance of weak approximation with BMO}
		For $K=\QQ$ and $L=\QQ(\sqrt{3}),$  the $3$-fold $X$ has the following properties.	
		\begin{itemize}
			\item $X$ has a $K$-rational point, and satisfies weak approximation with Brauer-Manin obstruction 
			off $\infty_K.$
			\item $X_L$ does not satisfy  weak approximation with Brauer-Manin obstruction 
			off $\infty_L.$
		\end{itemize}
	\end{prop}

	\begin{proof}
		This is the same as in the proof of Theorem \ref{theorem main result: non-invariance of weak approximation with BMO}.
	\end{proof}

	The $3$-fold $X$ that we constructed, has an affine open subvariety defined by the following equations, which is a closed subvariety of $\AA^5$ with affine coordinates $(x,y,z,x',y').$
	\begin{equation*}
		\begin{cases}
			y^2-73z^2=(1-x^2)(x^2-73)(x'-4)^2+(99x^2+1)(5428x^2/5329+1/5329)\\ 
			{y'}^2={x'}^3-16
		\end{cases}
	\end{equation*}
	
	%
	%

	\subsection{An explicit unconditional example for Theorem \ref{theorem main result: non-invariance of the Hasse principle with BMO for odd degree}}\label{subsection main example2}
	
	In the subsection,
	let $K=\QQ,$ and let $\zeta_7$ be a primitive $7$-th root of unity. Let $\alpha=\zeta_7+\zeta_7^{-1}$ with the minimal polynomial $x^3+x^2-2x-1.$ Let  $L=\QQ(\alpha).$  Then the degree $[L:K]=3.$  We will construct an explicit Ch\^atelet surface bundle over a curve having properties of Theorem \ref{theorem main result: non-invariance of the Hasse principle with BMO for odd degree}.

	\subsubsection{Choosing an elliptic curve} 
	
	Let $E$ be an elliptic curve defined over $\QQ$ by a homogeneous equation:
	$$w_1^2w_2=w_0^3-343w_0w_2^2-2401w_2^3$$ with homogeneous coordinates $(w_0:w_1:w_2)\in \PP^2.$ This elliptic curve defined over $\QQ,$ is of analytic rank $0,$ so it satisfies weak approximation with Brauer-Manin obstruction off $\infty_K.$ By computer calculation, we have the Mordell-Weil groups $E(K)=\{(0:1:0)\}$ and $E(L)=\{(7\alpha^2+14\alpha-7:0:1),(7\alpha^2-7\alpha-14:0:1),(-14\alpha^2-7\alpha+21:0:1),(0:1:0)\}.$ 
	So the triple $(E,K,L)$ is of type $I.$

	\subsubsection{Choosing a dominant morphism}
	Let $\PP^2\backslash\{(1:0:0)\}\to \PP^1$ be a morphism over $\QQ$ given by $(w_0:w_1:w_2)\mapsto(w_1:w_2).$ Composite with the natural inclusion $E\hookrightarrow \PP^2\backslash\{(1:0:0)\}.$ We get a morphism $\gamma\colon E\to \PP^1,$ which is a dominant morphism of degree $3.$ The morphism $\gamma$ maps $E(K)$ to $\{\infty\}=\{(1:0)\},$ and maps $E(L)\backslash E(K)$ to $\{0\}=\{(0:1)\}.$
	By B\'ezout's Theorem \cite[Chapter I. Corollary 7.8]{Ha97} or Hurwitz's Theorem \cite[Chapter IV. Corollary 2.4]{Ha97}, the branch locus of $\gamma$ is $\{(1:0)\}\bigcup \{(u_0:1)|27u_0^4 + 129654u_0^2 - 5764801=0\}.$

	\subsubsection{Construction of a Ch\^atelet surface bundle}

	Let $P_\infty(x)=14(x^4-89726),$ and let $P_0(x)=(x^2-878755181)(5x^2-4393775906).$ Notice that these polynomials $P_\infty$ and $P_0$ are separable.  As mentioned in Example \ref{example2: construction of V_infty} and Example \ref{example2: construction of V_0}, let $V_\infty$ be the Ch\^atelet surface given by $y^2-377z^2=P_\infty(x),$ and let $V_0$ be the Ch\^atelet surface given by $y^2-377z^2=P_0(x).$ 
	Let $\tilde{P}_\infty(x_0,x_1)$ and $\tilde{P}_0(x_0,x_1)$ be the homogenizations of $P_\infty$ and $P_0.$ Let $(u_0:u_1)\times(x_0:x_1)$ be the coordinates of $\PP^1\times\PP^1,$ and let $s'=u_0^2\tilde{P}_\infty(x_0,x_1)+u_1^2\tilde{P}_0(x_0,x_1)\in \Gamma(\PP^1\times\PP^1,\Ocal(1,2)^{\otimes2}).$ For $P_0(x)$ and $P_\infty(x)$ are coprime in $K[x],$ by Jacobian criterion, the locus $Z'$ defined by $s'=0$ in $\PP^1\times\PP^1$ is smooth. Then the branch locus of the composition $Z'\hookrightarrow \PP^1\times\PP^1  \stackrel{pr_1}\to\PP^1,$ denoted by $R,$ is finite and contained in $\PP^1\backslash \{(1:0)\}.$
	Let $B=E\times \PP^1.$ Let $\Lcal=(\gamma,id)^*\Ocal(1,2),$ and let $s=(\gamma,id)^* (s')\in \Gamma(B,\Lcal^{\otimes2}).$ With these notations, we have the following lemma.

	\begin{lemma}
		The curve $Z$ defined by $s=0$ in $B$ is smooth, projective, and geometrically connected.
	\end{lemma}
	
	\begin{proof}
		For smoothness of $Z,$ we need to check that the branch locus $R$ 
		does not intersect with the branch locus of  $\gamma\colon E\to \PP^1$.
		For $R$ is contained in $\PP^1\backslash \{(1:0)\},$
		we can assume the homogeneous coordinate $u_1=1,$ then the point in $R$ satisfies one of the following equations:  $
		14u_0^2 + 5=0,~44863u_0^2 - 137894762198231040=0,~70345184u_0^4 -216218987126801139936u_0^2+1=0.$ The polynomials of these equations are irreducible over $\QQ.$  By comparing these irreducible polynomials with the branch locus of $\gamma,$ we get the conclusion that these two branch loci do not intersect.
		The same argument as in the proof of Theorem \ref{theorem main result: non-invariance of weak approximation with BMO},	
		the locus $Z$ defined by $s=0$ in $B$ is geometrically connected. So it is smooth, projective, and geometrically connected.
	\end{proof}

	Let $X$ be the zero locus of $(1,-a,-s)\in\Gamma(B,\Ocal_B\oplus\Ocal_B\oplus\Lcal^{\otimes2})\subset\Gamma(B,\Sym^2\Ecal)$ in the projective space bundle $\Proj(\Ecal)$ with the natural projection $\alpha\colon  X\to B.$  By the same argument as in the proof of Theorem \ref{theorem main result: non-invariance of the Hasse principle with BMO for odd degree}, the variety $X$ is smooth, projective, and geometrically connected. Let $ \beta\colon X \to E$ be the composition of $\alpha$ and $pr_1.$  Then it is a Ch\^atelet surface bundle over the curve $E.$ For this $X,$ we have the following proposition.

	\begin{prop}\label{example2: main result: non-invariance of the Hasse principle with BMO for odd degree}
		For $K=\QQ$ and $L=\QQ(\zeta_7+\zeta_7^{-1}),$  the $3$-fold $X$ has the following properties.
		\begin{itemize}
			\item $X$ is a counterexample to the Hasse principle, and its failure of the Hasse principle is explained by the Brauer-Manin obstruction.
			\item $X_L$ is a counterexample to the Hasse principle, but its failure of the Hasse principle cannot be explained by the Brauer-Manin obstruction.
		\end{itemize}
	\end{prop}

	\begin{proof}
		This is the same as in the proof of Theorem \ref{theorem main result: non-invariance of the Hasse principle with BMO for odd degree}.
	\end{proof}

	The $3$-fold $X$ that we constructed, has an affine open subvariety defined by the following equations, which is closed subvariety of $\AA^5$ with affine coordinates $(x,y,z,x',y').$
	\begin{equation*}
		\begin{cases}
			y^2-377z^2=14(x^4-89726){y'}^2+(x^2-878755181)(5x^2-4393775906)\\ 
			{y'}^2={x'}^3-343x'-2401
		\end{cases}
	\end{equation*}

	\subsection{An explicit unconditional example for Theorem \ref{theorem main result: non-invariance of the Hasse principle with BMO for exist real place}}\label{subsection main example3}
	
	In the subsection,
	let $K=\QQ$ and $L=\QQ(\sqrt{3}).$ We will construct an explicit Ch\^atelet surface bundle over a curve having properties of Theorem \ref{theorem main result: non-invariance of the Hasse principle with BMO for exist real place}.

	\subsubsection{Choosing an elliptic curve and a dominant morphism} 
	Let $E$ and $\gamma\colon E\to \PP^1$ be the same as in Subsection \ref{subsection main example1}. For $E(\RR)$ is connected, the triple $(E,K,L)$ is of type $II.$

	\subsubsection{Construction of a Ch\^atelet surface bundle}

	Let $P_\infty(x)=5(x^4+805),$ and let $P_0(x)=-5(x^4+115).$ Notice that these polynomials $P_\infty$ and $P_0$ are irreducible.  As mentioned in Example \ref{example3: construction of V_infty} and Example \ref{example3: construction of V_0}, let $V_\infty$ be the Ch\^atelet surface given by $y^2+23z^2=P_\infty(x),$ and let $V_0$ be the Ch\^atelet surface given by $y^2+23z^2=P_0(x).$ 
	Let $\tilde{P}_\infty(x_0,x_1)$ and $\tilde{P}_0(x_0,x_1)$ be the homogenizations of $P_\infty$ and $P_0.$ Let $(u_0:u_1)\times(x_0:x_1)$ be the coordinates of $\PP^1\times\PP^1,$ and let $s'=u_0^2\tilde{P}_\infty(x_0,x_1)+u_1^2\tilde{P}_0(x_0,x_1)\in \Gamma(\PP^1\times\PP^1,\Ocal(1,2)^{\otimes2}).$ For $P_0(x)$ and $P_\infty(x)$ are coprime in $K[x],$ by Jacobian criterion, the locus $Z'$ defined by $s'=0$ in $\PP^1\times\PP^1$ is smooth. Then the branch locus of the composition $ Z'\hookrightarrow \PP^1\times\PP^1  \stackrel{pr_1}\to\PP^1,$ denoted by $R,$ is finite, and contained in $\PP^1\backslash \{(1:0)\}.$
	Let $B=E\times \PP^1.$ Let $\Lcal=(\gamma,id)^*\Ocal(1,2),$ and let $s=(\gamma,id)^* (s')\in \Gamma(B,\Lcal^{\otimes2}).$ With these notations, we have the following lemma.

	\begin{lemma}
		The curve $Z$ defined by $s=0$ in $B$ is smooth, projective, and geometrically connected.
	\end{lemma}
	
	\begin{proof}
		For smoothness of $Z,$ we need to check that the branch locus $R$ 
		does not intersect with the branch locus of  $\gamma\colon E\to \PP^1$.
		For $R$ is contained in $\PP^1\backslash \{(1:0)\},$
		we can assume the homogeneous coordinate $u_1=1,$ then the point in $R$ satisfies one of the following equations:
		$	u_0^2 -1=0,~7u_0^2 - 1=0.$ 	
		By comparing this locus with the branch locus of $\gamma,$ these two branch loci do not intersect.
		The same argument as in the proof of Theorem \ref{theorem main result: non-invariance of weak approximation with BMO},	
		the locus $Z$ defined by $s=0$ in $B$ is geometrically connected. So it is smooth, projective, and geometrically connected.
	\end{proof}

	Let $X$ be the zero locus of $(1,-a,-s)\in\Gamma(B,\Ocal_B\oplus\Ocal_B\oplus\Lcal^{\otimes2})\subset\Gamma(B,\Sym^2\Ecal)$ in the projective space bundle $\Proj(\Ecal)$ with the natural projection $\alpha\colon  X\to B.$  By the same argument as in the proof of Theorem \ref{theorem main result: non-invariance of the Hasse principle with BMO for odd degree}, the variety $X$ is smooth, projective, and geometrically connected. Let $ \beta\colon X \to E$ be the composition of $\alpha$ and $pr_1.$ Then it is a Ch\^atelet surface bundle over the curve $E.$ For this $X,$ we have the following proposition.

	\begin{prop}\label{example3: main result: non-invariance of the Hasse principle with BMO for even degree and exist real place for L}
		For $K=\QQ$ and $L=\QQ(\sqrt{3}),$  the $3$-fold $X$ has the following properties.
		\begin{itemize}
			\item $X$ is a counterexample to the Hasse principle, and its failure of the Hasse principle is explained by the Brauer-Manin obstruction.
			\item $X_L$ is a counterexample to the Hasse principle, but its failure of the Hasse principle cannot be explained by the Brauer-Manin obstruction.
		\end{itemize}
	\end{prop}

	\begin{proof}
		This is the same as in the proof of Theorem \ref{theorem main result: non-invariance of the Hasse principle with BMO for exist real place}.
	\end{proof}

	The $3$-fold $X$ that we constructed, has an affine open subvariety defined by the following equations, which is closed subvariety of $\AA^5$ with affine coordinates $(x,y,z,x',y').$
	\begin{equation*}
		\begin{cases}
			y^2+23z^2=5(x^4+805)(x'-4)^2-5(x^4+115)\\
			{y'}^2={x'}^3-16
		\end{cases}
	\end{equation*}

	\subsection{Exceptions}\label{subsection main example4}
	For Question \ref{question on HP}, when the degree $[L:K]$ is even and $L$ has no real place, besides cases of Theorem \ref{theorem main result: non-invariance of the Hasse principle with BMO for odd degree} and Theorem \ref{theorem main result: non-invariance of the Hasse principle with BMO for exist real place}, we can give some unconditional examples, case by case, to give negative answers to Question \ref{question on HP}, although we do not have a uniform way to construct them. In this subsection, we give an explicit example to explain how it works for
	the case that $K=\QQ$ and $L=\QQ(i).$

	\subsubsection{Choosing an elliptic curve} 
	Let $E$ be an elliptic curve defined over $\QQ$ by a homogeneous equation:
	$$w_1^2w_2=w_0^3-16w_2^3$$ with homogeneous coordinates $(w_0:w_1:w_2)\in \PP^2.$ This is an elliptic curve with complex multiplication. Its quadratic twist $E^{(-1)}$ is isomorphic to an elliptic curve defined by a homogeneous equation:
	$w_1^2w_2=w_0^3+16w_2^3$ with homogeneous coordinates $(w_0:w_1:w_2)\in \PP^2.$  These elliptic curves $E$ and $E^{(-1)}$ defined over $\QQ,$ are  of analytic rank $0.$ Then the Tate-Shafarevich group $\Sha(E,\QQ)$ is finite, so $E$ satisfies weak approximation with Brauer-Manin obstruction off $\infty_K.$ The Mordell-Weil groups $E(K)$ and $E^{(-1)}(K)$ are finite, so $E(L)$ is finite. Indeed, the Mordell-Weil group $E(K)=\{(0:1:0)\}$ and $E(L)=\{(0:\pm 4i:1),(0:1:0)\}.$ So the triple $(E,K,L)$ is of type $I.$
	

	\subsubsection{Choosing a dominant morphism}

	Let $\PP^2\backslash\{(1:0:0)\}\to \PP^1$ be a morphism over $\QQ$ given by $(w_0:w_1:w_2)\mapsto(w_1:4w_2).$ Composite with the natural inclusion $E\hookrightarrow \PP^2\backslash\{(1:0:0)\},$ then we get a morphism $\gamma\colon E\to \PP^1,$ which is a dominant morphism of degree $3.$ The dominant morphism $\gamma$ maps $E(K)$ to $\{\infty\}=\{(1:0)\},$ and maps $(0:\pm 4i:1)$ to $(\pm i:1).$ By B\'ezout's Theorem \cite[Chapter I. Corollary 7.8]{Ha97} or Hurwitz's Theorem \cite[Chapter IV. Corollary 2.4]{Ha97}, the branch locus of $\gamma$ is $\{(1:0),(\pm i:1)\}.$

	\subsubsection{Properties of Ch\^atelet surfaces}

	Let $P_\infty(x)=2(x^4-10x^2+15),~ P_0(x)=-2(5x^4-39x^2+75),$ and let $P_1(x)=P_\infty(x)+iP_0(x).$ Notice that all those polynomials $P_\infty,P_0,P_1$ are separable, and $P_1(x)=-2[x^2-(5+i)][(-1+5i)x^2-15i].$ The two polynomials $x^2-(5+i)$ and $(-1+5i)x^2-15i$ are irreducible over $\QQ(i)$ (indeed, they are irreducible over $\QQ(i)_3$). 
	Let $V_\infty$ be the Ch\^atelet surface over $\QQ$ given by $y^2+15z^2=P_\infty(x),$ and let $V_1$ be the Ch\^atelet surface over $\QQ(i)$ given by $y^2+15z^2=P_1(x).$ With these notations, we have the following lemmas.

	\begin{lemma}\label{example4 V_infty has no local point}
		The Ch\^atelet surface $V_\infty$ given by $y^2+15z^2=2(x^4-10x^2+15),$ has a $\QQ_v$-point for all $v\in \Omega_{\QQ}\backslash\{5\},$ but no $\QQ_5$-point. 
	\end{lemma} 
	
	\begin{proof}
		Suppose that $v=\infty_\QQ.$ Let $x_0=0.$ Then $(-15,P_\infty(x_0))_v=(-15,30)_v=1,$ which implies that the surface $V_\infty^0$ admits a $\RR$-point with $x=0.$\\
		Suppose that $v=2.$ For $-15\equiv 1\mod 8,$ by Hensel's lemma, the element $-15$ is a square in $\QQ_2.$ By Remark \ref{remark birational to PP^2}, the surface $V_\infty$ admits a $\QQ_2$-point.\\
		Suppose that $v=3.$ Let $x_0=2.$ Then $(-15,P_\infty(x_0))_v=(-15,-18)_v=(-15,9)_v(-15,-2)_v=1,$ which implies that the surface $V_\infty^0$ admits a $\QQ_3$-point with $x=2.$\\
		Suppose that $v\in\Omega_{\QQ}^f\backslash \{2,3,5\}.$ Take $x_0\in \QQ_v$ such that the valuation $v(x_0)<0.$ Then by Lemma \ref{lemma Hensel lemma for Hilbert symbal}, we have $(-15,P_\infty(x_0))_v=(-15,2)_v=1,$ which implies that $V_\infty^0$ admits a $\QQ_v$-point with $x=x_0.$\\
		Suppose that $v=5.$ Then $(-15,2)_v=-1.$  Let $x\in \QQ_5.$ If $v(x)\leq 0,$ then by Lemma \ref{lemma Hensel lemma for Hilbert symbal}, we have $(-15,P_\infty(x))_v=(-15,2x^4)_v=-1.$ If $v(x)>0,$ then  by Lemma \ref{lemma Hensel lemma for Hilbert symbal}, we have $(-15,P_\infty(x))_v=(-15,30)_v=-1.$ In each case, we have $(-15,P_\infty(x))_v=-1,$ which implies that $V_\infty^0$ has no $\QQ_5$-point. By Remark \ref{remark the implicit function thm and local constant Brauer group}, we have $V_\infty(\QQ_5)=\emptyset.$ 
	\end{proof}

	\begin{lemma}\label{example4 V_1 exists local point}
		For any $v'\in \Omega_{\QQ(i)},$  the Ch\^atelet surface $V_1$ given by $y^2+15z^2=-2[x^2-(5+i)][(-1+5i)x^2-15i],$ has a  $\QQ(i)_{v'}$-point. 
	\end{lemma} 
	
	\begin{proof}
		For the only archimedean place is complex, we only need to consider finite places.\\
		Suppose that $v'$ is a $2$-adic place. For $-15\in \QQ_2^{\times 2},$ by Remark \ref{remark birational to PP^2}, the surface $V_1$ admits a $\QQ(i)_{v'}$-point.\\
		Suppose that $v'=3.$ For $-2\in \QQ_3^{\times 2},$ we have $(-15,-2)_{v'}=1.$ By Lemma \ref{lemma Hensel lemma for Hilbert symbal}, we have $(-15,(-4-i)(-1-10i))_{v'}=(-15,(1+i)^2)_{v'}.$
		Let $x_0=1.$ Then  $(-15,P_1(x_0))_{v'}=(-15,-2(-4-i)(-1-10i))_{v'}=(-15,(1+i)^2)_{v'}=1,$ which implies that $V_1^0$ admits a $\QQ(i)_3$-point with $x=1.$\\
		Suppose that $v'|5,$ then $\QQ(i)_{v'}\cong\QQ_5.$ By Lemma \ref{lemma Hensel lemma for Hilbert symbal}, we have $(-15,-2(-5+i)(-10-17i))_{v'}=(-15,-34)_{v'}.$
		Let $x_0=1+i.$  Then $(-15,P_1(x_0))_{v'}=(-15,-2(-5+i)(-10-17i))_{v'}=(-15,-34)_{v'}=1,$ which implies that $V_1^0$ admits a $\QQ(i)_{v'}$-point with $x=1+i.$\\
		Suppose that $v'|13,$ then $\QQ(i)_{v'}\cong\QQ_{13}.$ Let $x_0=1.$ By Lemma \ref{lemma hilbert symbal lifting for odd prime}, we have  $(-15,P_1(x_0))_{v'}=(-15,-2(-4-i)(-1-10i))_{v'}=1,$ which implies that $V_1^0$ admits a $\QQ(i)_{v'}$-point with $x=1.$\\
		Suppose that $v'\in\Omega_{\QQ(i)}^f\backslash \{2,3,5,13\}.$ Take $x_0\in \QQ_{v'}$ such that the valuation $v'(x_0)<0.$ By Lemma \ref{lemma Hensel lemma for Hilbert symbal}, we have $(-15,P_\infty(x_0))_{v'}=(-15,-2(-1+5i)x_0^4)_{v'}.$  By Lemma \ref{lemma hilbert symbal lifting for odd prime}, we have $(-15,-2(-1+5i))_{v'}=1,$ so $(-15,P_\infty(x_0))_{v'}=1,$ which implies that $V_1^0$ admits a $\QQ(i)_{v'}$-point with $x=x_0.$
	\end{proof}
	
	For $P_1(x)$ is a product of two degree-2 irreducible factors, according to \cite[Proposition 7.1.1]{Sk01}, the Brauer group $\Br(V_1)/\Br(\QQ(i))\cong \ZZ/2\ZZ.$ Furthermore, by Proposition 7.1.2 in loc. cit, we take the quaternion algebra $A=(-15,(-1+5i)x^2-15i)\in \Br(V_1)$ as a generator element of this group. Then we have the equality $A=(-15,(-1+5i)x^2-15i)=(-15,-2(x^2-(5+i)))$ in $\Br(V_1).$  With these notations, we have the following lemmas.

	\begin{lemma}\label{example4 the valuation of Brauer group on local points are fixed}
		For any $v'\in \Omega_{\QQ(i)}$ and any $P_{v'}\in V_1(\QQ(i)_{v'}),$
		\begin{equation*}
			\inv_{v'}(A(P_{v'}))=\begin{cases}
				0& if\quad v'\neq 3\\
				1/2 & if \quad v'=3
			\end{cases}
		\end{equation*}
	\end{lemma}
	
	\begin{proof}
		By Remark \ref{remark the implicit function thm and local constant Brauer group}, it suffices to compute the local invariant $\inv_{v'}(A(P_{v'}))$ for all $P_{v'}\in V_1^0(\QQ(i)_{v'}).$
		
		Suppose that $v'$ is an archimedean place or a $2$-adic place or a $181$-adic place.  Then $-15\in\QQ(i)_{v'}^{\times 2},$ so $\inv_{v'}(A(P_{v'}))=0$ for all $P_{v'}\in V_1(\QQ(i)_{v'}).$\\
		Suppose that $v'|5.$ Let $x\in \QQ(i)_{v'}.$ If $v'(x)\leq 0,$ then by Lemma \ref{lemma Hensel lemma for Hilbert symbal}, we have $(-15,(-1+5i)x^2-15i)_{v'}=(-15,-x^2)_{v'}=1.$ If $v'(x)>0,$ then by Lemma \ref{lemma Hensel lemma for Hilbert symbal}, we have $(-15,-2(x^2-(5+i)))_{v'}=(-15,2i)_{v'}=1.$ So $\inv_{v'}(A(P_{v'}))=0$ for all $P_{v'}\in V_1^0(\QQ(i)_{v'}).$\\ 
		Suppose that $v'\in \Omega_{\QQ(i)}^f\backslash \{3,$ all $2$-adic, $5$-adic and $181$-adic places$\}.$ Take an arbitrary $P_{v'}\in V_1^0(\QQ(i)_{v'}).$ If $\inv_{v'}(A(P_{v'}))=1/2,$ then $(-15,(-1+5i)x^2-15i)_{v'}=-1=(-15,-2(x^2-(5+i)))_{v'}$ at $P_{v'}.$ By Lemma \ref{lemma hilbert symbal lifting for odd prime}, the first and last equalities imply that $v'((-1+5i)x^2-15i)$ and $v'(x^2-(5+i))$ are odd, so they are positive. Hence $v'((-1+5i)x^2-15i-(-1+5i)(x^2-(5+i)))=v'(-10+9i)>0.$ But $v'\nmid181,$ which is a contradiction. So $\inv_{v'}(A(P_{v'}))=0.$
		
		Suppose that $v'=3.$  Take an arbitrary $P_{v'}\in V_1^0(\QQ(i)_{v'}).$ If $\inv_{v'}(A(P_{v'}))=1,$ then $(-15,(-1+5i)x^2-15i)_{v'}=1=(-15,-2(x^2-(5+i)))_{v'}$ at $P_{v'}.$ For $(-15,-1+5i)_{v'}=-1,$ the first equality implies that $v'(x)>0.$ Then by Lemma \ref{lemma Hensel lemma for Hilbert symbal}, we have
		$(-15,x^2-(5+i))_{v'}=(-15,-(5+i))_{v'}=-1,$ so $(-15,-2(x^2-(5+i)))_{v'}=(-15,-2)_{v'}(-15,-5-i)_{v'}=-1,$ which is a contradiction. So $\inv_{v'}(A(P_{v'}))=\half.$		
	\end{proof}

	\begin{lemma}\label{example4 V_1 has no rational point}
		The Ch\^atelet surface $V_1$ has no $\QQ(i)$-rational point. 
	\end{lemma} 
	
	\begin{proof}
		If there exists $\QQ(i)$-rational point $P,$ by the global reciprocity law $\sum_{v\in \Omega_{\QQ(i)}}\inv_v(A(P))=0$ in $\QQ/\ZZ.$ 
		But from Lemma \ref{example4 the valuation of Brauer group on local points are fixed}, this sum is $1/2,$  which is nonzero in $\QQ/\ZZ.$ So $V_1$ has no $\QQ(i)$-rational point.
	\end{proof}

	\subsubsection{Construction of a Ch\^atelet surface bundle}

	Let $\tilde{P}_\infty(x_0,x_1)$ and $\tilde{P}_0(x_0,x_1)$ be the homogenizations of $P_\infty$ and $P_0.$ Let $(u_0:u_1)\times(x_0:x_1)$ be the coordinates of $\PP^1\times\PP^1,$ and let $s'=(u_0^2+2u_1^2)\tilde{P}_\infty(x_0,x_1)+u_0u_1\tilde{P}_0(x_0,x_1)\in \Gamma(\PP^1\times\PP^1,\Ocal(1,2)^{\otimes2}).$ By Jacobian criterion, the locus $Z'$ defined by $s'=0$ in $\PP^1\times\PP^1$ is smooth.
	Then the branch locus of the composition $ Z'\hookrightarrow \PP^1\times\PP^1  \stackrel{pr_1}\to\PP^1$ is finite, and contained in $\PP^1\backslash \{(1:0),(\pm i,1)\}.$
	Let $B=E\times \PP^1.$ Let $\Lcal=(\gamma,id)^*\Ocal(1,2),$ and let $s=(\gamma,id)^* (s')\in \Gamma(B,\Lcal^{\otimes2}).$ With these notations, we have the following lemma.

	\begin{lemma}
		The curve $Z$ defined by $s=0$ in $B$ is smooth, projective, and geometrically connected.
	\end{lemma}
	
	\begin{proof}
		For the branch locus of the composition $Z'\hookrightarrow \PP^1\times\PP^1  \stackrel{pr_1}\to\PP^1$ is contained in $\PP^1\backslash \{(1:0),(\pm i,1)\},$ and the branch locus of $\gamma\colon E\to \PP^1$ is $\{(1:0),(\pm i,1)\},$ they do not intersect, which implies the smoothness of $Z.$
		The same argument as in the proof of Theorem \ref{theorem main result: non-invariance of weak approximation with BMO},	
		the locus $Z$ defined by $s=0$ in $B$ is geometrically connected. So it is smooth, projective, and geometrically connected.
	\end{proof}

	Let $X$ be the zero locus of $(1,-a,-s)\in\Gamma(B,\Ocal_B\oplus\Ocal_B\oplus\Lcal^{\otimes2})\subset\Gamma(B,\Sym^2\Ecal)$ in the projective space bundle $\Proj(\Ecal)$ with the natural projection $\alpha\colon  X\to B.$  By the same argument as in the proof of Theorem \ref{theorem main result: non-invariance of the Hasse principle with BMO for odd degree}, the variety $X$ is smooth, projective, and geometrically connected. Let $ \beta\colon X \to E$ be the composition of $\alpha$ and $pr_1.$ Then it is a Ch\^atelet surface bundle over the curve $E.$ For this $X,$ we have the following proposition.

	\begin{prop}\label{example4: non-invariance of the Hasse principle with BMO for degree=2 and no real place}
		For $K=\QQ$ and $L=\QQ(i),$  the $3$-fold $X$ has the following properties.
		\begin{itemize}
			\item $X$ is a counterexample to the Hasse principle, and its failure of the Hasse principle is explained by the Brauer-Manin obstruction.
			\item $X_L$ is a counterexample to the Hasse principle, but its failure of the Hasse principle cannot be explained by the Brauer-Manin obstruction.
		\end{itemize}
	\end{prop}

	\begin{proof}
		Let $\sigma$ be the generator element of Galois group $\Gal(L/K).$
		We will show $X(\AA_K)\neq \emptyset.$
		By our construction, each $K$-rational fiber of $\beta$ is isomorphic to $V_\infty.$ By Lemma \ref{example4 V_infty has no local point}, the set $V_\infty(\AA_K^{\{5\}})\neq \emptyset.$ So $X(\AA_K^{\{5\}})\neq \emptyset.$ For $5$ splits completely in $L,$ take a place $v'\in \Omega_L$ above $5,$ i.e. $v'|5$ in $L.$ Then $\QQ_{5}\cong L_{v'}.$ By Lemma \ref{example4 V_1 exists local point}, the set $V_1(\AA_L)\neq \emptyset.$ 	
		Since $V_1\bigsqcup \sigma(V_1)\cong \bigsqcup_{P\in E(L)\backslash E(K) }\beta^{-1}(P)\subset X_L,$
		 the set  $X(\QQ_5)=X(L_{v'})\supset V_1(L_{v'})\neq \emptyset.$ So $X(\AA_K)\neq \emptyset.$	\\ 
		We will show $X(\AA_K)^{\Br}=\emptyset.$
		For $E(K)=pr^{\infty_K}(C(\AA_K)^{\Br}),$ the functoriality of Brauer-Manin pairing implies  
		$pr^{\infty_K}(X(\AA_K)^{\Br})\subset \bigcup_{P\in  E(K) }\beta^{-1}(P) (\AA_K^{\infty_K}).$ By Lemma \ref{example4 V_infty has no local point}, the set $V_\infty(\QQ_5)=\emptyset,$ so
		$pr^{\infty_K}(X(\AA_K)^{\Br})\subset \bigcup_{P\in  E(K) }\beta^{-1}(P) (\AA_K^{\infty_K})\cong V_\infty(\AA_K^{\infty_K})=\emptyset,$ which implies that $ X(\AA_K)^{\Br}=\emptyset.$\\
		So, the variety $X$ is a counterexample to the Hasse principle, and its failure of the Hasse principle is explained by the Brauer-Manin obstruction.
		
		We will show $X_L(\AA_L)^{\Br}\neq \emptyset.$
		By Proposition \ref{Poonen's main proposition}, the map $\alpha_L^*\colon\Br(C_L)\to \Br(X_L)$ is an isomorphism. By the functoriality of Brauer-Manin pairing, the set  $X_L(\AA_L)^{\Br}$ contains $V_1(\AA_L),$ which is nonempty.\\
		We will show $X(L)=\emptyset.$
		By Lemma \ref{example4 V_1 has no rational point}, the set $V_1(L)=\emptyset,$ so also $\sigma(V_1)(L)=\emptyset.$ For $5$ splits completely in $L,$ the emptiness of $V_\infty(\QQ_5)$ implies $V_\infty(\AA_L)=\emptyset.$ Since $X(L)\subset\bigsqcup_{P\in E(L) }\beta^{-1}(P)(L)\cong V_\infty(L)\bigcup V_1(L) \bigcup \sigma(V_1)(L),$ the set $X(L)=\emptyset.$ \\
		So, the variety $X_L$ is a counterexample to the Hasse principle, but its failure of the Hasse principle cannot be explained by the Brauer-Manin obstruction.
	\end{proof}

	The $3$-fold $X$ that we constructed, has an affine open subvariety defined by the following equations, which is closed subvariety of $\AA^5$ with affine coordinates $(x,y,z,x',y').$
	\begin{equation*}
		\begin{cases}
			y^2+15z^2=(x^4-10x^2+15)({y'}^2+32)/8-(5x^4-39x^2+75)y'/2\\ 
			{y'}^2={x'}^3-16
		\end{cases}
	\end{equation*}

	\begin{footnotesize}
		\noindent\textbf{Acknowledgements.} The author would like to thank my thesis advisor Y. Liang for proposing the related problems, papers and many fruitful discussions.  The author was partially supported by NSFC Grant No. 12071448.
	\end{footnotesize}


\end{document}